\documentclass[a4paper,11pt]{article}
\pdfoutput=1

\usepackage[sec]{modstyle}
\title{On the tensor product of enriched $\infty$-categories}
\author{Rune Haugseng}

\date{\today}

\newcommand{\Algdpsh}{\Algd_{/\catname{PSh}}}
\newcommand{\LSpc}{\widehat{\Spc}}
\newcommand{\LPSh}{\widehat{\PSh}}
\renewcommand{\Seg}{\txt{Seg}}
\renewcommand{\FunV}{\FUN^{\uV}}
\begin{document}

\maketitle
\begin{abstract}
  We show that the tensor product of $\infty$-categories enriched in a
  suitable monoidal $\infty$-category preserves colimits in each
  variable, fixing a mistake in an earlier paper of Gepner and the
  author. We also prove that essentially surjective and fully faithful
  functors form a factorization system on enriched
  $\infty$-categories, and that the tensor product and internal hom
  are compatible with this.
\end{abstract}

\tableofcontents

\section{Introduction}
If $\uV$ is a symmetric monoidal category,
then we can define a
tensor product of $\uV$-enriched categories: for $\uV$-categories
$\eA$ and $\eB$ with sets of objects $S$ and $T$, respectively, their tensor product
$\eA \otimes \eB$ has objects $S \times T$, and Hom objects
\[ (\eA \otimes \eB)((s,t), (s',t')) \cong \eA(s,s') \otimes \eB(t,t').\]

The \icatl{} version of this tensor product was defined in \cite{enr},
where Gepner and the author introduced enriched \icats{}, but there is
an unfortunate mistake in our treatment of the tensor product: In
\cite{enr}*{Proposition 5.7.16} we claim that the functor that assigns
to a presentably monoidal \icat{} $\uV$ the \icat{} $\Cat(\uV)$ of
$\uV$-enriched \icats{} is lax monoidal with respect to the cocomplete
tensor product. This means, in particular, that if $\uV$ is
presentably symmetric monoidal, then so is $\Cat(\uV)$, \ie{} the
tensor product of $\uV$-\icats{} preserves colimits in each variable.
However, we do not actually give a proof of this! To explain the
issue, we need to unpack the construction a bit:
\begin{itemize}
\item The \icat{} $\Cat(\uV)$ is obtained as a localization of a larger \icat{} $\Algcat(\uV)$ of \emph{categorical algebras}.\footnote{In the body of the paper we will instead refer to these objects as \emph{algebroids}. Other names used in the literature are \emph{$\uV$-precategories} and \emph{flagged $\uV$-\icats{}}.}
The tensor product is obtained by restricting one on $\Algcat(\uV)$, and \cite{enr}*{Proposition 5.7.16} follows from the analogous statement for $\Algcat(\uV)$ \cite{enr}*{Corollary 4.3.16}.
\item The desired colimit-preserving lax monoidal structure map \[\Algcat(\uV) \otimes \Algcat(\uV') \to \Algcat(\uV \otimes \uV')\] corresponds to a functor \[\Algcat(\uV) \times \Algcat(\uV') \to \Algcat(\uV \otimes \uV')\] that preserves colimits in each variable. The latter is defined as the composite
  \begin{equation}
    \label{eq:algcattens}
  \Algcat(\uV) \times \Algcat(\uV') \xto{\boxtimes} \Algcat(\uV \times \uV') \xto{\mu_{*}} \Algcat(\uV \otimes \uV')
  \end{equation}
  of an ``external product'' $\boxtimes$ and a functor given by
  composing with a canonical monoidal functor
  $\mu \colon \uV \times \uV' \to \uV \otimes \uV'$.
\item In the proof of \cite{enr}*{Corollary 4.3.16}, the statement
  that \cref{eq:algcattens} preserves colimits in each variable is
  asserted to follow from \cite{enr}*{Proposition 4.3.15}, which says
  that the external product preserves colimits in each variable, and
  \cite{enr}*{Proposition 4.3.14}, which says that if
  $F \colon \uU \to \uU'$ is a colimit-preserving monoidal functor, then the induced functor
  \[ F_{*} \colon \Algcat(\uU) \to \Algcat(\uU') \]
  preserves colimits.
\item However, this argument
  does not work since\footnote{Let us note that there are also issues with the proof of
  \cite{enr}*{Proposition 4.3.15}: This relies on \cite{enr}*{Lemma
    3.6.15}, but writing out the relevant colimit formulas for
  operadic left Kan extensions we see that this can only be true under
  strong cofinality assumptions.}
  the functor $\mu$ certainly does \emph{not}
  preserve colimits! (Instead, it preserves colimits \emph{in each
  variable}.)
\end{itemize}
The goal of this note is to give a correct proof that
\cref{eq:algcattens} preserves colimits in each variable, for any cocomplete
monoidal \icats{} $\uV$ and $\uV'$ whose tensor products preserve small
colimits in each variable.

\subsection*{Overview}
In \S\ref{sec:algds} we recall the definition of enriched \icats{} via
algebroids (or categorical algebras), and in \S\ref{sec:algdtens} we
review the definition of tensor products of enriched \icats{} in this
setting. To understand the behaviour of the tensor product we will
make use of an alternative model of \icats{} enriched in presheaves
with Day convolution using ``Segal
presheaves'', which we recall in \S\ref{sec:segalpsh}. We can also
define a tensor product in this setting, and we relate this to the one
for algebroids in \S\ref{sec:comp}. We then discuss inner anodyne maps
in Segal presheaves in \S\ref{sec:innan} and use these to prove our
main result in the case of enrichment in presheaves in
\S\ref{sec:tenspsh}. Building on this we extend the result to
presentable \icats{} in \S\ref{sec:tenspres} and then to general
cocomplete \icats{} in \S\ref{sec:tensgen}. Finally, in
\S\ref{sec:ffes} we show that the essentially surjective and the fully
faithful functors form a factorization system on enriched \icats{}
that is compatible with the tensor product.

\subsection*{Notation}
We write $\Spc$ for the \icat{} of spaces, $\MonCatI$ for that of monoidal \icats{} (and strong monoidal functors), and $\MonCatIlax$ for that of monoidal \icats{} and lax monoidal functors. We often denote the unit of a monoidal \icat{} by $\bbone$.

\subsection*{Acknowledgments}
I thank Fernando Abell\'an Garc\'ia and David Gepner for helpful conversations about this paper.
I also thank Bastiaan Cnossen and Adrian Clough for the proof of
\cref{lem:pbgivesmono} and Louis Martini for inspiring that of \cref{cor:FunFF}.

\section{Enriched \icats{} as algebroids}\label{sec:algds}
In this section we will briefly review the main definitions of enriched
\icats{} from \cite{enr}. We assume the
reader is already familiar with the algebraic framework of
(generalized) non-symmetric \iopds{}; see \cite{enr}*{\S 2} for
motivation and \cite{enr}*{\S 3} for a detailed discussion of these
objects\footnote{But note that here we will denote the \icat{} of
  algebras for a \gnsiopd{} $\uO$ in a monoidal \icat{} $\uV$ as just
  $\Alg_{\uO}(\uV)$.}.

\begin{defn}
  For $X \in \Spc$, let $\DopX \to \Dop$ denote the left fibration for
  the functor $\Dop \to \Spc$ obtained from $X$ by right Kan extension
  along the inclusion $\{[0]\} \hookrightarrow \Dop$. Then
  $(\DopX)_{[n]} \simeq X^{\times (n+1)}$, and $\DopX$ is a double
  \icat{} (and so in particular a \gnsiopd{}). An algebra for $\DopX$
  we call a \emph{$\Dop$-algebroid}\footnote{In \cite{enr}, these were
    called \emph{categorical algebras}.}, or in this paper just
  \emph{algebroid}, with space of objects $X$. If $\uC^{\otimes}$ is a
  monoidal \icat{}, we will write
  $\Algd(\uC) = \Algd_{\Dop}(\uC) \to \Spc$ for the cartesian
  fibration corresponding to the functor
  $X \mapsto \Alg_{\DopX}(\uC)$.
\end{defn}

\begin{defn}
  Let $\Alg_{/\OpdIgns} \to \OpdIgns \times \OpdIgns$ denote the orthofibration (in the sense of \cite{HHLN1}) for the functor \[\Alg_{(\blank)}(\blank) \colon (\OpdIgns)^{\op} \times \OpdIgns \to \CatI.\] We can then define $\Algd \to \Spc \times \MonCatIlax$ as the pullback
  \[
    \begin{tikzcd}
      \Algd \arrow{r} \arrow{d} & \Alg_{\OpdIgns} \arrow{d} \\
      \Spc \times \MonCatIlax \arrow{r} & \OpdIgns \times \OpdIgns.
    \end{tikzcd}
  \]
  In other words, $\Algd \to \MonCatIlax$ is the cocartesian fibration for the functor $\uC^{\otimes} \mapsto \Algd(\uC)$; we write $F_{*} \colon \Algd(\uC)\to \Algd(\uD)$ for the functor induced by composition with a lax monoidal functor $F \colon \uC \to \uD$.
\end{defn}

Next, we want to define enriched \icats{} as a full subcategory of algebroids; this requires introducing some notation:

\begin{defn}
  For a space $S$, let $E_{S}$ denote the unique algebroid with $S$ as
  space of objects in the terminal \gnsiopd{} $\Dop$ (which is also
  the unique monoidal structure on the terminal category); this gives
  an equivalence $E_{\blank} \colon \Spc \isoto \Algd(*)$. For any
  monoidal \icat{} $\uC^{\otimes}$, the unit gives the unique monoidal
  functor $u_{\uC} \colon \Dop \to \uC^{\otimes}$. We write
  $E_{S,\uC} := u_{\uC,*}E_{S}$ (and we denote this also by $E_{S}$
  when $\uC$ is clear from the context), so
  $E_{S,\uC}(x,y) \simeq \bbone$ for all $x,y \in S$. When $S$ is the
  set $\{0,\ldots,n\}$ we write $E^{n}$ for $E_{S}$, so that we have a
  cosimplicial object $E^{\bullet} \colon \simp \to \Algd(\uC)$ for
  any monoidal \icat{} $\uC$. Let $\eA \in \Algd(\uC)$ be an algebroid
  with underlying space $X$; we write
  \[\iota_{n}\eA := \Map_{\Algd(\uC)}(E^{n}, \eA),\] so that
  $\iota_{\bullet}\eA$ is a simplicial space, and we define
  $\iota \eA$ to be its colimit. Here $\iota_{0}\eA$ is the underlying
  space $X$, and we say that $\eA$ is \emph{complete} if the canonical
  map $\iota_{0}\eA \to \iota \eA$ is an equivalence. Equivalently (by
  \cite{enr}*{Corollary 5.2.10}), $\eA$ is complete \IFF{} it is local
  with respect to the degeneracy map $E^{1} \to E^{0}$.
\end{defn}

\begin{remark}
  We refer to \cite{enr}*{\S 5.2} for further discussion of this definition. In particular, note that $\eA$ is complete \IFF{} its underlying $\Spc$-enriched \icat{}, \ie{} its image under the lax monoidal functor $\Map_{\uC}(\bbone,\blank)$, corresponds to a Segal space that is complete in the sense of Rezk~\cite{RezkCSS}.
\end{remark}

\begin{notation}
  We denote the full subcategory of $\Algd(\uC)$ spanned by the complete objects by $\Cat(\uC)$; we will also refer to its objects as \emph{$\uC$-enriched \icats{}}. We also write $\catname{Enr}$ for the full subcategory of $\Algd$ spanned by the complete algebroids in all monoidal \icats{}.
\end{notation}

\begin{observation}
  By \cite{enr}*{Theorem 5.6.6} the full subcategory $\Cat(\uC)$ is a
  localization of $\Algd(\uC)$ for any $\uC$. From this it follows
  (\cite{enr}*{Proposition 5.7.1}) that the restricted projection
  $\Enr \to \MonCatIlax$ is still a cocartesian fibration, and the
  inclusion $\Enr \hookrightarrow \Algd$ has a left adjoint that
  preserves cocartesian morphisms. (In other words, the cocartesian
  transport functor $\Cat(\uC) \to \Cat(\uD)$ over
  $F \colon \uC \to \uD$ is given by composing with $F$ and then
  completing in $\Algd(\uD)$.) Note, however, that the other
  projection $\Enr \to \Spc$ is no longer a cartesian fibration (but
  we will see in \S\ref{sec:ffes} that it has cartesian morphisms over
  monomorphisms in $\Spc$).
\end{observation}

\begin{defn}
  We say a morphism $F \colon \ec{A} \to
  \ec{B}$ in $\Algd(\uC)$ is \emph{fully faithful} if for all objects $x,y \in
  \ec{A}$, the induced map
  \[ \eA(x,y) \to \eB(Fx,Fy)\] is an equivalence in $\uC$, and
  \emph{essentially surjective} if the morphism of spaces
  $\iota \eA \to \iota \eB$ is surjective on $\pi_{0}$.
\end{defn}

\begin{observation}
  By \cite{enr}*{Lemma 5.3.2}, the fully faithful morphisms are precisely the cartesian morphisms for the projection $\Algd(\uC) \to \Spc$.
\end{observation}

\begin{remark}
  By \cite{enr}*{Theorem 5.6.6}, the full subcategory $\Cat(\uC)$ is
  the localization of $\Algd(\uC)$ with respect to the fully faithful
  and essentially surjective morphisms. In particular, a morphism
  $F \colon \eA \to \eB$ between \emph{complete} algebroids is an
  equivalence \IFF{} it is fully faithful and essentially surjective.
\end{remark}

\section{The tensor product via algebroids}\label{sec:algdtens}
In this section we recall the tensor products of algebroids and
of enriched \icats{}, as defined in \cite{enr}. Following
\cite{HeineEnr}, we will define this by internalizing the cartesian
product in an \icat{} of algebroids in varying monoidal \icats{},
starting with an observation about cartesian products in fibrations:
\begin{propn}\label{propn:orthfibprod}
  Suppose $\uA$ and $\uB$ are \icats{} with finite products, and
  suppose $F \colon \uA^{\op} \times \uB \to \CatI$ is a functor such
  that $F(a,\blank) \colon \uB \to \CatI$ preserves products for all
  $a \in \uA$. Then if $\uE \to \uA \times \uB$ is the corresponding
  orthofibration, the \icat{} $\uE$ has cartesian products, and these
  are preserved by the projections to $\uA$ and $\uB$. Specifically,
  for $u \in F(a,b)$ and $v \in F(a',b')$, their cartesian product in
  $\uE$ is the image of $(u,v)$ under the composite functor
  \[
    \begin{tikzcd}[column sep=3cm]
      F(a,b) \times F(a',b') \arrow{r}{F(\pi_{a},b) \times F(\pi_{a'}, b)} &
      F(a \times a', b) \times F(a \times a', b') \arrow{d}{(F(a \times a', \pi_{b}), F(a \times a', \pi_{b'}))^{-1}}\\
       & F(a \times a', b \times b');
    \end{tikzcd}
  \]
  in other words, we have
  \[ \pi_{b,!}(u \times v) \simeq \pi_{a}^{*}u, \qquad \pi_{b',!}(u \times v) \simeq \pi_{a'}^{*}v.\]
\end{propn}
\begin{proof}
  Let $q$ be the object of $\uE$ that we claim is the product of $u$ and $v$. Then we must show that composition with the maps
  \[ q \to \pi_{b,!}q \simeq \pi_{a}^{*}u \to u, \qquad q \to \pi_{b',!}q \simeq \pi_{a'}^{*}v \to v\]
  induces an equivalence
  \[ \Map_{\uE}(x, q) \isoto \Map_{\uE}(x, u) \times \Map_{\uE}(x, v)\]
  for every object $x \in \uE$ over $(s,t) \in \uA \times \uB$.
  We note that this composition map lies over the equivalence
  \[ \Map_{\uA}(s, a \times a') \times \Map_{\uB}(t, b \times b') \isoto \Map_{\uA}(s,a) \times \Map_{\uB}(t,b) \times \Map_{\uA}(s,a') \times \Map_{\uB}(t, b').\]
  Given maps $\alpha = (\alpha_{1},\alpha_{2}) \colon s \to a \times
  a'$ and $\beta = (\beta_{1},\beta_{2}) \colon t \to b \times b'$, it
  suffices to show that we get an equivalence on fibres over $(\alpha,\beta)$. Since $\uE \to \uB$ is a cocartesian fibration, we can identify
  \[
    \begin{split}
      \Map_{\uE}(x,q)_{\beta} & \simeq \Map_{\uE_{b \times b'}}(\beta_{!}x, q) \\ & \simeq \Map_{\uE_{b}}(\beta_{1,!}x, \pi_{b,!}q) \times \Map_{\uE_{b'}}(\beta_{2,!}x, \pi_{b',!}q) \\
       & \simeq \Map_{\uE}(x, \pi_{a}^{*}u)_{\beta_{1}} \times \Map_{\uE}(x, \pi_{a'}^{*}v)_{\beta_{2}},
    \end{split}
  \]
  where the second equivalence uses that $F$ preserves products in the second variable, and the composite equivalence is given by composition with our cocartesian morphisms. Next we use that $\uE \to \uA$ is a cartesian fibration to conclude that composition with our cartesian morphisms gives equivalences
  \[ \Map_{\uE}(x, \pi_{a}^{*}u)_{\alpha} \simeq \Map_{\uE}(x, u)_{\alpha_{1}}, \qquad \Map_{\uE}(x, \pi_{a'}^{*}v)_{\alpha} \simeq \Map_{\uE}(x,v)_{\alpha_{2}}.\]
  Taking fibres over $(\alpha,\beta)$ we can combine these two equivalences to see that our composition map indeed gives an equivalence
  \[ \Map_{\uE}(x, q)_{(\alpha,\beta)} \simeq \Map_{\uE}(x,u)_{(\alpha_{1},\beta_{1})} \times\Map_{\uE}(x,v)_{(\alpha_{2},\beta_{2})}, \]
  as required.
\end{proof}

\begin{cor}\label{algdcart}
  The \icat{} $\Algd$ has cartesian products. The product of $\eA \colon \DopX \to \uC^{\otimes}$ and $\eB \colon \Dop_{Y} \to \uD^{\otimes}$ is the algebroid
  \[ \eA \boxtimes \eB \colon \Dop_{X \times Y} \simeq \Dop_{X} \times_{\Dop} \Dop_{Y} \xto{\eA \times_{\Dop} \eB} \uC^{\otimes} \times_{\Dop} \uD^{\otimes} \simeq (\uC \times \uD)^{\otimes}. \]
\end{cor}
\begin{proof}
  Apply \cref{propn:orthfibprod} to the functor $\Alg_{\Dop_{(\blank)}}(\blank) \colon \Spc \times \MonCatIlax \to \CatI$ (which preserves products in the second variable since $\Alg_{\uO}(\blank)$ is right adjoint to the tensoring of the \gnsiopd{} $\uO$ with \icats{}).
\end{proof}

To see that the product on $\Algd$ restricts to $\Enr$ we make some observations:

\begin{lemma}\label{extprodcomp}\
  \begin{enumerate}[(i)]
  \item For $\eA \in \Algd(\uC)$ and $\eB \in \Algd(\uD)$, we have a natural equivalence
    \[ \iota_{\bullet}(\eA \boxtimes \eB) \simeq \iota_{\bullet}(\eA) \times \iota_{\bullet}(\eB)\]
    of simplicial spaces.
  \item If $\eA \in \Algd(\uC)$ and $\eB \in \Algd(\uD)$ are both
    complete, then $\eA \boxtimes \eB$ is also complete.
  \item If $F \colon \eA \to \eA'$ and $G \colon \eB \to \eB'$ are both either fully faithful or essentially surjective in $\Algd(\uC)$ and $\Algd(\uD)$, respectively, then so is
    \[ F \boxtimes G \colon  \eA \boxtimes \eB \to \eA' \boxtimes \eB' \] in
    $\Algd(\uC \times \uD)$.
  \end{enumerate}
\end{lemma}
\begin{proof}
  For any space $S$, we have a natural commutative square
  \[
    \begin{tikzcd}
      \Map_{\Algd}(E_{S}, \eA \boxtimes \eB) \arrow{r}{\sim} \arrow{d} & \Map_{\Algd}(E_{S}, \eA) \times \Map_{\Algd}(E_{S}, \eB) \arrow{d} \\
      \Map_{\MonCatIlax}(*, \uC \times \uD) \arrow{r}{\sim} & \Map_{\MonCatIlax}(*, \uC) \times \Map_{\MonCatIlax}(*, \uD),
    \end{tikzcd}
  \]
  where the horizontal morphisms are equivalences since
  $\eA \boxtimes \eB$ is a product in $\Algd$. Taking fibres over the
  unit in $\uC \times \uD$, we get a natural equivalence
  \[ \Map_{\Algd(\uC \times \uD)}(E_{S,\uC \times \uD}, \eA \boxtimes
    \eB) \simeq \Map_{\Algd(\uC)}(E_{S,\uC}, \eA) \times
    \Map_{\Algd(\uD)}(E_{S,\uD}, \eB).\] This gives in particular the
  equivalence of simplicial spaces
  \[\iota_{\bullet}(\eA \boxtimes \eB) \simeq \iota_{\bullet}(\eA)
  \times \iota_{\bullet}(\eB)\] required for (i). It is then immediate that $\eA \boxtimes \eB$ is
  complete if $\eA$ and $\eB$ are so. Moreover, for the essential surjectivity case in (iii) we have
  \[ \iota(F \boxtimes G) \simeq \iota F \times \iota G,\] which is indeed surjective on $\pi_{0}$ if $F$ and $G$ are essentially surjective. For the fully faithful case we note that for $(x,y)$ and $(x',y')$ in $\eA \boxtimes \eB$ the induced map
  \[ (\eA \boxtimes \eB)((x,y), (x',y'))  \to (\eA' \boxtimes \eB')((F \boxtimes G)(x,y), (F \boxtimes G)(x',y'))\]
  is the tensor product of the maps for $F$ and $G$, so this is an equivalence if these are both fully faithful.
\end{proof}

Combining \cref{algdcart} and \cref{extprodcomp}, we have:
\begin{cor}
  The \icat{} $\Enr$ has cartesian products, and these are preserved by the inclusion in $\Algd$ and by the localization $\Algd \to \Enr$. \qed
\end{cor}

We can now internalize these external products:
\begin{cor}\label{cor:algdOmonoidal}
  Suppose $\uC$ is an $\uO \times \Dop$-monoidal \icat{}. Then $\Algd(\uC)$ has a canonical $\uO$-monoidal structure, given as the pullback
  \[
    \begin{tikzcd}
      \Algd(\uC)^{\otimes} \arrow{r} \arrow{d} & \Algd^{\times} \arrow{d} \\
      \uO \arrow{r}{\uC} & (\MonCatIlax)^{\times},
    \end{tikzcd}
  \]
  where the bottom horizontal map views the $\uO \times \Dop$-monoidal \icat{} $\uC$ as an $\uO$-algebra in monoidal \icats{}. 
\end{cor}

\begin{observation}
  In particular, if $\uC$ is symmetric (or $E_{n+1}$-)monoidal then
  $\Algd(\uC)$ is symmetric (or $E_{n}$)monoidal; the tensor product of $\eA \colon \DopX \to \uC^{\otimes}$ and $\eB \colon \Dop_{Y} \to \uC^{\otimes}$ is the composite
  \[ \Dop_{X \times Y} \simeq \Dop_{X} \times_{\Dop} \Dop_{Y} \xto{\eA
      \times_{\Dop} \eB} \uC^{\otimes} \times_{\Dop} \uC^{\otimes} \to
    \uC^{\otimes},\] where the last map is the tensor product on $\uC$
  viewed as a commutative (or $E_{n}$-)algebra in monoidal \icats{}.
\end{observation}

\begin{remark}
  In fact, to get an $\uO$-monoidal structure on $\Algd(\uC)$, it
  suffices for $\uC$ to be an $\uO$-monoid in $\MonCatIlax$, which
  means that $\uC$ has both an $\uO$-monoidal structure and a monoidal
  structure that are compatible in a lax sense. Such ``duoidal''
  structures on \icats{} have been studied by Torii~\cite{ToriiDuoid}.
\end{remark}

\begin{cor}
    Suppose $\uC$ is an $\uO \times \Dop$-monoidal \icat{}. Then $\Cat(\uC)$ has a canonical $\uO$-monoidal structure, given as the pullback
  \[
    \begin{tikzcd}
      \Cat(\uC)^{\otimes} \arrow{r} \arrow{d} & \Enr^{\times} \arrow{d} \\
      \uO \arrow{r}{\uC} & \MonCatIlax.
    \end{tikzcd}
  \]
  Moreover, the localization $\Algd(\uC) \to \Cat(\uC)$ is an $\uO$-monoidal functor, with the inclusion $\Cat(\uC) \hookrightarrow \Algd(\uC)$ as its lax $\uO$-monoidal right adjoint. \qed
\end{cor}

Our aim in this paper is to prove that if $\uC$ is an $\uO \otimes \mathbb{E}_{1}$-monoidal \icat{} that is compatible with small colimits, then so is the induced $\uO$-monoidal structures on $\Algd(\uC)$ and $\Cat(\uC)$. More generally, if $\mu \colon \uC \times \uD \to \uE$ is a monoidal functor that preserves colimits in each variable, we want to show that the induced functor
\[ \Algd(\uC) \times \Algd(\uD) \xto{\boxtimes} \Algd(\uC \times \uD)
  \xto{\mu_{*}} \Algd(\uE)\] preserves colimits in each variable, and
similarly for its restriction to complete algebroids.

\begin{observation}\label{obs:compprescolim}
  If $F \colon \uC \to \uD$ is a monoidal functor that preserves
  colimits, then the induced functor
  $F_{*}\colon \Algd(\uC) \to \Algd(\uD)$ preserves colimits by
  \cite{enr}*{Proposition 3.6.10}, as does the functor
  $F_{*} \colon \Cat(\uC) \to \Cat(\uD)$ by \cite{enr}*{Lemma
    5.7.7}. Now recall that the tensor product of cocomplete \icats{}
  \cite{HA}*{\S 4.8.1} provides for cocomplete \icats{} $\uC, \uD$ a
  universal functor $\uC \times \uD \to \uC \otimes \uD$ that
  preserves colimits in each variable: composition with this gives an
  equivalence between functors $\uC \times \uD \to \uE$ that preserve
  colimits in each variable and functors $\uC \otimes \uD \to \uE$
  that preserve colimits. To prove our desired result it therefore
  suffices to prove the universal case, namely that the functor
  \[ \Algd(\uC) \times \Algd(\uD) \to \Algd(\uC \times \uD) \to
    \Algd(\uC \otimes \uD)\] preserves colimits in each variable, as
  does the corresponding functor on $\Cat(\blank)$.
\end{observation}

We will take a rather indirect route to this result, however: since colimits in algebroids are complicated to describe, we will instead consider an alternative description of enriched \icats{} as certain presheaves satisfying Segal conditions.

\section{Enriched \icats{} as Segal presheaves}\label{sec:segalpsh}
In this section we recall the alternative description of enriched
\icats{} as ``Segal presheaves'', first introduced in
\cite{enr}*{\S 4.5}.

\begin{notation}
  Suppose $\uC^{\otimes} \to \Dop$ is a monoidal \icat{}. We write
  $p \colon \uC_{\otimes}\to \simp$ for the cartesian fibration corresponding
  to the same functor as this; note that the opposite $(\uC_{\otimes})^{\op} \to \Dop$ is the cocartesian fibration for the canonical monoidal structure on $\uC^{\op}$, so we will also write $\uC^{\op,\otimes} := (\uC_{\otimes})^{\op}$. An object of $\uC_{\otimes}$ over $[n]$
  that corresponds to $\bfc = (c_{1},\ldots,c_{n})$ under the
  equivalence $\uC_{\otimes,[n]} \simeq \uC^{\times n}$ will be denoted
  $[n](\bfc)$; given a morphism $\phi \colon [m] \to [n]$ in $\simp$ and
  an object $[n](\bfc)$ of $\uC_{\otimes,[n]}$, we will write
  \[ \phi^{\bfc} \colon [m](\phi^{*}\bfc) \to [n](\bfc)\]
  for a cartesian morphism over $\phi$; here we have
  \[ (\phi^{*}\bfc)_{j} \simeq \bigotimes_{\phi(j-1) < i \leq \phi(j)}
    c_{i}.\] We will also write $\Delta^{n}(\bfc)$ for the presheaf on
  $\uC_{\otimes}$ represented by $[n](\bfc)$.
\end{notation}

\begin{observation}
  In $\PSh(\uC_{\otimes})$ we have a pullback square
  \[
    \begin{tikzcd}
      \Delta^{m}(\phi^{*}\bfc) \arrow{r} \arrow{d} & \Delta^{n}(\bfc) \arrow{d} \\
      p^{*}\Delta^{m} \arrow{r} & p^{*}\Delta^{n},
    \end{tikzcd}
  \]
  as a special case of \cite{enropd}*{Lemma 2.7.10}.
\end{observation}

\begin{defn}
  Let $\Delta^{n}_{\Seg}$ denote the spine of $\Delta^{n}$, that is
  the iterated pushout
  $\Delta^{1} \amalg_{\Delta^{0}} \cdots \amalg_{\Delta^{0}}
  \Delta^{1}$, which maps to $\Delta^{n}$ via the inert maps
  $[0],[1] \to [n]$. Then we define $\Delta^{n}_{\Seg}(\bfc)$ as the
  pullback
  \[
    \begin{tikzcd}
      \Delta^{n}_{\Seg}(\bfc) \arrow{r} \arrow{d} & \Delta^{n}(\bfc) \arrow{d} \\
      p^{*}\Delta^{n}_{\Seg} \arrow{r} & p^{*}\Delta^{n}.
    \end{tikzcd}
  \]
  Since $p^{*}$ preserves colimits and $\PSh(\uC_{\otimes})$ is locally cartesian closed, we see that
  \[ \Delta^{n}_{\Seg}(\bfc) \simeq \Delta^{1}(c_{1}) \amalg_{\Delta^{0}} \cdots \amalg_{\Delta^{0}} \Delta^{1}(c_{n}).\]
  An object $\Phi \in \PSh(\uC_{\otimes})$ is then called a \emph{Segal presheaf} if it is local with respect to all maps $\Delta^{n}_{\Seg}(\bfc) \to \Delta^{n}(\bfc)$. In other words, $\Phi$ is a Segal presheaf if all of the maps
  \[ \Phi([n](\bfc)) \to \Phi([1](c_{1})) \times_{\Phi([0])} \cdots \times_{\Phi([0])} \Phi([1](c_{n}))\]
  are equivalences. We write $\PSeg(\uC_{\otimes})$ for the full
  subcategory of $\PSh(\uC_{\otimes})$ spanned by the Segal presheaves
  and $L_{\Seg}$ for the localization functor.
\end{defn}

\begin{remark}
  The Segal presheaves are precisely the Segal
  $\uC^{\op,\otimes}$-spaces in the terminology of \cite{patterns3}.
\end{remark}

\begin{observation}
  Let $\Phi \colon \uC_{\otimes} \to \uD_{\otimes}$ be an oplax monoidal functor, that is a functor over $\simp$ that preserves inert cartesian morphisms. Then for the induced adjunction
  \[ \Phi_{!} : \PSh(\uC_{\otimes}) \rightleftarrows \PSh(\uD_{\otimes}) : \Phi^{*}\]
  on presheaves, we have that the image of the Segal morphism $\Delta^{n}_{\Seg}(\bfc) \to \Delta^{n}(\bfc)$ is precisely the Segal morphism
  \[ \Delta^{n}_{\Seg}(\Phi(\bfc)) \to \Delta^{n}(\Phi(\bfc)),\] since
  $\Phi_{!}$ restricts to $\Phi$ on representable presheaves and
  preserves colimits; the presheaf $\Phi_{!}(\Delta^{n}(\bfc))$ is
  represented by $\Phi([n](\bfc))$, which is $[n](\Phi(\bfc))$ where
  $\Phi(\bfc)_{i}= \Phi(c_{i})$, since $\Phi$ preserves inert cartesian
  morphisms. It follows that $\Phi$ induces an adjunction
  \[ L_{\Seg}\Phi_{!} \colon \PSeg(\uC_{\otimes}) \rightleftarrows \PSeg(\uD_{\otimes}) : \Phi^{*}\]
  on Segal presheaves.
\end{observation}

Although Segal presheaves are functorial in oplax monoidal functors, to compare this functoriality to that for algebroids we would need to enter into some $(\infty,2)$-categorical considerations. Since this is not necessary to prove our desired result on tensor products, we will restrict ourselves to (strong) monoidal functors:
\begin{defn}
  Let $\PSeg \to \MonCatI$ be the fibration for the functor
  $\uC \mapsto \PSeg(\uC_{\otimes})$; this is a cocartesian fibration
  for the covariant functoriality and cartesian for the contravariant
  functorality of $\PSeg(\blank)$ in monoidal functors. 
\end{defn}

\begin{observation}\label{obs:psegviafib}
  We can describe the \icat{} $\PSeg$ explicitly in terms of right
  fibrations, as follows: Let $\RFib$ denote the full subcategory of
  $\Ar(\CatI)$ spanned by the right fibrations; then
  $\ev_{1} \colon \RFib \to \CatI$ is the fibration for the functor
  taking $\uC$ to $\PSh(\uC)$. We can define $\PSeg$ as the full
  subcategory of the pullback $\MonCatI \times_{\CatI} \RFib$ of
  $\ev_{1}$ and $(\blank)_{\otimes} \colon \MonCatI \to \CatI$ spanned
  by the right fibrations that correspond to Segal presheaves. Here we
  can identify $\MonCatI$ (in terms of cartesian fibrations) with a
  subcategory of $\Cat_{\infty/\simp}$, so that $\PSeg$ corresponds to
  a subcategory of
  \[\Ar(\CatI) \times_{\CatI} \Cat_{\infty/\simp} \simeq
    \Ar(\Cat_{\infty/\simp}).\]
\end{observation}

We can now introduce the (external) tensor product of Segal presheaves:

\begin{notation}
  Let $\uC^{\otimes}, \uD^{\otimes}$ be monoidal \icats{}. For objects $X \in \PSh(\uC_{\otimes}), Y \in \PSh(\uD_{\otimes})$ we write $X \boxtimes Y$ for the composite
  \[ (\uC \times \uD)^{\op,\otimes} := (\uC_{\otimes} \times_{\simp}\uD_{\otimes})^{\op} \to (\uC_{\otimes})^{\op} \times (\uD_{\otimes})^{\op} \xto{X \times Y} \Spc \times \Spc \to \Spc, \]
  where the last functor is the cartesian product.
\end{notation}

\begin{observation}\label{obs:boxlimcolim}
  The external product is characterized by the equivalence
  \[
    \begin{split}
      \Map(\Delta^{n}(\bfc, \bfd), X \boxtimes Y) & \simeq (X \boxtimes Y)([n](\bfc, \bfd)) \\
                                                  & \simeq X([n], \bfc) \times Y([n], \bfd) \\
      & \simeq \Map(\Delta^{n}(\bfc), X) \times \Map(\Delta^{n}(\bfd), Y).
    \end{split}
  \]
  Since colimits in functor \icats{} are computed pointwise, and the cartesian product of spaces preserves colimits in each variable, we see that the external product
  \[ \blank \boxtimes \blank \colon \PSh(\uC_{\otimes}) \times
    \PSh(\uD_{\otimes}) \to \PSh(\uC_{\otimes} \times_{\simp}
    \uD_{\otimes})\] preserves colimits in each variable. Moreover, it also preserves limits in $\PSh(\uC_{\otimes}) \times
    \PSh(\uD_{\otimes})$ (\ie{}, diagrams that are limits in both variables are taken to limits by $\boxtimes$).
\end{observation}

\begin{lemma}\label{lem:boxpresseg}
  If $X \in \PSeg(\uC_{\otimes})$ and $Y \in \PSeg(\uD_{\otimes})$, then the presheaf $X \boxtimes Y$ is a Segal presheaf on $(\uC \times \uD)_{\otimes}$.
\end{lemma}
\begin{proof}
  We must show that $X \boxtimes Y$ is local with respect to $\Delta^{n}_{\Seg}(\bfc,\bfd) \to \Delta^{n}(\bfc, \bfd)$. But we have a commutative square
  \[
    \begin{tikzcd}
      \Map(\Delta^{n}(\bfc, \bfd), X \boxtimes Y) \arrow{r}{\sim} \arrow{d} & \Map(\Delta^{n}(\bfc), X) \times \Map(\Delta^{n}(\bfd), Y) \arrow{d} \\
      \Map(\Delta_{\Seg}^{n}(\bfc, \bfd), X \boxtimes Y) \arrow{r}{\sim}  & \Map(\Delta_{\Seg}^{n}(\bfc), X) \times \Map(\Delta_{\Seg}^{n}(\bfd), Y), 
    \end{tikzcd}
  \]
  where the right vertical map is an equivalence if $X$ and $Y$ are Segal presheaves. Then the left vertical map is also an equivalence, as required.
\end{proof}

\begin{propn}
  The external tensor product is the cartesian product on the \icat{} $\PSeg$, and this is preserved by the projection to $\MonCatI$.  
\end{propn}
\begin{proof}
  To see this, consider the full subcategory of
  $\Ar(\Cat_{\infty/\simp})$ spanned by the right fibrations. This
  obviously has a cartesian product, given by taking pullbacks over
  $\simp$, and if $\uE \to \uB$ and $\uE' \to \uB'$ are right
  fibrations where $\uB,\uB'$ live over $\simp$, then their cartesian
  product $\uE \times_{\simp} \uE' \to \uB \times_{\simp} \uB'$ is
  precisely the right fibration corresponding to our external
  product. Via the description of $\PSeg$ from \cref{obs:psegviafib}, we can identify this with a subcategory of $\Ar(\Cat_{\infty/\simp})$, which we know is closed under cartesian products from \cref{lem:boxpresseg}.
\end{proof}

Applying the same pullback construction as in \cref{cor:algdOmonoidal}, we get:
\begin{cor}
  If $\uC$ is an $\uO \otimes E_{1}$-monoidal \icat{}, then
  $\PSeg(\uC_{\otimes})$ has an $\uO$-monoidal structure. \qed
\end{cor}

\section{Comparison for presheaf enrichment}\label{sec:comp}
In this section we first recall the equivalence between Segal
presheaves on $\uC_{\otimes}$ and algebroids valued in presheaves on
$\uC$ equipped with Day convolution; this was first discussed in
\cite{enr}*{\S 4.5}, and is also a special case of the results of
\cite{patterns3}*{\S 2.5}. We then explicate the relation between the cartesian product on $\PSeg$ and that on $\Algd$ under this correspondence.

First, recall that for any small
monoidal \icat{} $\uC^{\otimes}$ we can define a \emph{Day
  convolution} monoidal structure on $\PSh(\uC)$. This can be defined
as the pullback
\[
  \begin{tikzcd}
    \PSh(\uC)^{\otimes} \arrow{r} \arrow{d} & \RFib^{\times} \arrow{d}{\ev_{1}} \\
    \Dop \arrow{r}{\uC} & \CatI^{\times},
  \end{tikzcd}
\]
where $\RFib$ is the full subcategory of $\Ar(\CatI)$ spanned by the right fibrations and the bottom horizontal map is the algebra in $\CatI$ corresponding to the monoidal \icat{} $\uC^{\otimes}$. The Day convolution has the universal property that for any \gnsiopd{} $\uO$ we have a natural equivalence
\[ \Alg_{\uO}(\PSh(\uC)) \simeq \Mon_{\uO \times_{\Dop}
    \uC^{\op,\otimes}}(\Spc).\]
We refer to \cite{patterns2}*{\S 6} for a
discussion of this construction (but note that it is originally due to
Heine~\cite{HeineThesis}), or see \cite{GlasmanDay} and \cite{HA}*{\S 2.2.6} for alternative approaches. In particular, we have
\[ \Alg_{\Dop_{X}}(\PSh(\uC)) \simeq \Mon_{\uC^{\op,\otimes}_{X}}(\Spc),\] where 
$\uC^{\op,\otimes}_{X} := \DopX \times_{\Dop} \uC^{\op,\otimes}$. From this it is easy to prove the following, which is also a special case of \cite{patterns3}*{Corollary 2.5.2}:
\begin{propn}
  There is an equivalence of \icats{} \[\PSeg(\uC_{\otimes}) \simeq \Algd(\PSh(\uC))\]
  for every small monoidal \icat{} $\uC$. \qed
\end{propn}

Moreover, this equivalence is natural in monoidal functors:
\begin{defn}
  Let $\Algdpsh$ be the \icat{} defined by the pullback
  \[
    \begin{tikzcd}
      \Algdpsh \arrow{r} \arrow{d} & \Algd \arrow{d} \\
      \MonCatI \arrow{r}{\PSh_{!}} & \LMonCatI.
    \end{tikzcd}
  \]
\end{defn}

\begin{propn}
  There is an equivalence of \icats{}
  \[ \Algdpsh \simeq \PSeg \]
  over $\MonCatI$.
\end{propn}
\begin{proof}
  This is a special case of \cite{patterns3}*{Corollary 2.6.12}.
\end{proof}

In the previous section we saw that the cartesian product on $\PSeg$ induced monoidal structures on Segal presheaves, and we now want to relate this to the monoidal structures we constructed in \S\ref{sec:algdtens}. For this we need to understand how the cartesian product on $\Algdpsh$ relates to that on $\Algd$.

\begin{observation}
  A functor $\PSh(\uC) \times \PSh(\uD) \to \uX$ that preserves
  colimits in each variable is determined by its restriction to
  $\uC \times \uD$, and so is the same thing as a colimit-preserving
  functor $\PSh(\uC \times \uD) \to \uX$. This shows that
  $\PSh(\uC \times \uD)$ is the cocomplete tensor product of
  $\PSh(\uC)$ and $\PSh(\uD)$. Thus, there is a canonical functor
  $\mu \colon \PSh(\uC) \times \PSh(\uD) \to \PSh(\uC \times \uD)$ that witnesses
  this universal property of the target. Explicitly, this is the
  functor that takes $(F,G)$ to the presheaf $F \times G$ given by
  $(c,d) \mapsto F(c) \times G(d)$: indeed, this functor preserves
  colimits in each variable (since the cartesian product in $\Spc$
  preserves these in each variable, and colimits in presheaves are
  computed objectwise), and restricts on representables to the Yoneda
  embedding for $\uC \times \uD$, since
  \[\Map_{\uC \times \uD}((\blank, \blank), (c,d)) \simeq
    \Map_{\uC}(\blank, c) \times \Map_{\uD}(\blank, d).\]
  We also have a functor $\pi \colon \PSh(\uC \times \uD) \to \PSh(\uC) \times \PSh(\uD)$, given by left Kan extension along the two projections from $\uC \times \uD$ to $\uC$ and $\uD$, and we claim that $\pi$ is the left adjoint of $\mu$. In terms of right fibrations, if $\uF \to \uC$ and $\uG \to \uD$ are the fibrations for $F$ and $G$, then  $\mu(F,G)$ corresponds to $\uF \times \uG \to \uC \times \uD$. Given a right fibration $\uE \to \uC \times \uD$ corresponding to $E \in \PSh(\uC \times \uD)$, then we have
  \[ \Map_{\PSh(\uC \times \uD)}(E, \mu(F,G)) \simeq \Map_{/\uC \times
      \uD}(\uE, \uF \times \uG) \simeq \Map_{/\uC}(\uE, \uF) \times
    \Map_{/\uD}(\uE, \uG).\] Here we can identify the right-hand side
  with $\Map_{\PSh(\uC) \times \PSh(\uD)}(\pi E, (F,G))$, since the
  functors given by left Kan extensions along the projections
  correspond on fibrations to composing with the projections and then forcing the result to
  be a right fibration.
\end{observation}

\begin{observation}
  If $\uC$ and $\uD$ are monoidal \icats{}, then $\pi \colon \PSh(\uC \times \uD) \to \PSh(\uC) \times \PSh(\uD)$ is a monoidal functor with respect to Day convolution (since left Kan extension along a monoidal functor gives a monoidal functor on Day convolution, and the cartesian product is the limit also in monoidal \icats{}). It follows that the right adjoint $\mu$ is lax monoidal.
\end{observation}

\begin{lemma}\label{lem:algdadj}
  If $L \colon \uC \to \uD$ is a monoidal functor with (lax monoidal) right adjoint $R \colon \uD \to \uC$, then composition with $L$ and $R$ gives an adjunction
  \[ L_{*} : \Algd(\uC) \rightleftarrows \Algd(\uD) : R_{*}.\]
\end{lemma}
\begin{proof}
  Immediate from the definition, since $\Algd(\blank)$ is a functor of \itcats{}.
\end{proof}

\begin{propn}
  The \icat{} $\Algdpsh$ has finite cartesian products which are preserved by the projections to $\MonCatI$ and $\Spc$. The product of $\eA \in \Alg_{\Dop_{X}}(\PSh(\uC))$ and $\eB \in \Alg_{\Dop_{Y}}(\PSh(\uD))$ is given by
  \[ \mu_{*}(\eA \boxtimes \eB) \colon \Dop_{X \times Y} \xto{\eA \times_{\Dop} \eB} \PSh(\uC)^{\otimes} \times_{\Dop} \PSh(\uD)^{\otimes} \xto{\mu} \PSh(\uC \times \uD)^{\otimes}.\]
\end{propn}
\begin{proof}
  Given $\ec{X} \in \Alg_{\Dop_{Z}}(\PSh(\uV))$, we must show that the map
  \[ \Map_{\Algdpsh}(\ec{X}, \mu_{*}(\eA \boxtimes \eB)) \to \Map_{\Algdpsh}(\ec{X}, \eA) \times \Map_{\Algdpsh}(\ec{X}, \eB),\]
  induced by composition with the cocartesian morphisms
  \[ \mu_{*}(\eA \boxtimes \eB) \to \pi_{1,!,*}\mu_{*}(\eA \boxtimes \eB) \simeq \eA,\] \[\mu_{*}(\eA \boxtimes \eB) \to \pi_{2,!,*}\mu_{*}(\eA \boxtimes \eB) \simeq \eB,\]
  is an equivalence. This map fits in a commutative square
  \[
    \begin{tikzcd}
      \Map(\ec{X}, \mu_{*}(\eA \boxtimes \eB)) \arrow{r} \arrow{d} & \Map(\ec{X}, \eA) \times \Map(\ec{X}, \eB) \arrow{d} \\
      \Map(\uV, \uC \times \uD) \arrow{r}{\sim} &  \Map(\uV, \uC) \times \Map(\uV, \uD),
    \end{tikzcd}
  \]
  so it suffices to check that the map on fibres over $(\Phi,\Psi) \colon \uV \to \uC \times \uD$ is an equivalence. The left-hand fibre we can identify with
  \[ \Map_{\Algd(\PSh(\uC \times \uD))}((\Phi_{!},\Psi_{!})_{*}\ec{X}, \mu_{*}(\eA \boxtimes \eB)).\]
  By \cref{lem:algdadj} we have an adjunction $\pi_{*} \dashv \mu_{*}$ on algebroids, so this is equivalent to
  \[ \Map_{\Algd(\PSh(\uC) \times \PSh(\uD))}(\pi_{*}(\Phi_{!},\Psi_{!})_{*}\ec{X}, \eA \boxtimes \eB) \simeq \Map_{\Algd(\PSh(\uC))}(\Phi_{!}\ec{X}, \eA) \times \Map_{\Algd(\PSh(\uD))}(\Psi_{!}\ec{X}, \eB),\]
  just as required.
\end{proof}

\begin{cor}\label{cor:psegprodinalgd}
  Under the equivalence of Segal presheaves and algebroids in presheaves, the external product
  \[ \PSeg(\uC_{\otimes}) \times \PSeg(\uD_{\otimes}) \to \PSeg((\uC \times \uD)_{\otimes})\]
  corresponds to the composite
  \[ \Algd(\PSh(\uC)) \times \Algd(\PSh(\uD)) \xto{\boxtimes} \Algd(\PSh(\uC)\times \PSh(\uD)) \xto{\mu_{*}} \Algd(\PSh(\uC \times \uD)). \]
\end{cor}

We saw in \cref{obs:compprescolim} that we are interested in proving that for cocomplete monoidal \icats{} $\uU$ and $\uV$, the composite 
\[ \Algd(\uV) \times \Algd(\uW) \to \Algd(\uV \times \uW) \to \Algd(\uV \otimes \uW) \]
preserves colimits in each variable. We have now seen that in the case of presheaves this is equivalent to the external product on Segal presheaves preserving colimits in each variable. In the next two sections we will prove the latter statement, and we will then deduce the general one from this.

\section{Inner anodyne maps}\label{sec:innan}
Let $\uC$ be a monoidal \icat{}, and let $p \colon \uC_{\otimes} \to \simp$ be the corresponding cartesian fibration. In this section we will discuss \emph{inner anodyne maps} in
$\PSh(\uC_{\otimes})$, closely following \cite{enropd}*{\S 2.7}, where the corresponding
notion was discussed in an operadic context.

\begin{defn}
  Let $\uC$ be a cocomplete $\infty$-category.  We say a class of
  morphisms in $\uC$ is \emph{weakly saturated} if it is closed under
  cobase change, transfinite compositions, and retracts. Given a class
  $\mathbb{S}$ of morphisms in $\uC$, we can consider the smallest
  weakly saturated class that contains $\mathbb{S}$, which we call the
  weakly saturated class \emph{generated} by $\mathbb{S}$.
\end{defn}

\begin{defn}\label{def partical/lambda x}
  For $[n](\bfc) \in \uC_{\otimes}$ we define a presheaf $\Lambda^{n}_{k}(\bfc)$ by the pullback
  \[
    \begin{tikzcd}
      \Lambda^{n}_{k}(\bfc) \arrow{r} \arrow{d} & \Delta^{n}(\bfc) \arrow{d} \\
      p^{*}\Lambda^{n}_{k} \arrow{r} & p^{*}\Delta^{n},
    \end{tikzcd}
  \]
  where the right vertical map is the unit morphism
  $\Delta^{n}(\bfc) \to p^{*}p_{!}\Delta^{n}(\bfc) \simeq p^{*}\Delta^{n}$. Similarly, we define $\partial \Delta^{n}(\bfc)$ by
  the pullback
  \[
    \begin{tikzcd}
      \partial \Delta^{n}(\bfc) \arrow{r} \arrow{d} & \Delta^{n}(\bfc) \arrow{d} \\
      p^{*}\partial\Delta^{n} \arrow{r} & p^{*}\Delta^{n}.
    \end{tikzcd}
  \]
  We refer to the inclusions
  $\Lambda^{n}_{k}(\bfc) \to \Delta^{n}(\bfc)$ for $0 < k < n$ as the
  \emph{inner horn inclusions} in $\PSh(\uC_{\otimes})$, and we say a
  map is \emph{inner anodyne} if it lies in the weakly saturated class
  generated by the inner horn inclusions.
\end{defn}

\begin{defn}
  We say a morphism in $\PSh(\uC_{\otimes})$ is a \emph{Segal equivalence} if its image under the localization functor to $\PSeg(\uC_{\otimes})$ is an equivalence. Equivalently, a morphism $f \colon X \to Y$ is a Segal equivalence \IFF{} composition with $f$ gives an equivalence $\Map(Y, Z) \to \Map(X, Z)$ for every Segal presheaf $Z$, or \IFF{} $f$ lies in the strongly saturated class generated by the spine inclusions $\Delta^{n}_{\Seg}(\bfc) \to \Delta^{n}(\bfc)$.
\end{defn}

Our goal is then to prove that inner anodyne maps in $\PSh(\uC_{\otimes})$ are Segal equivalences. The starting point for this is the following result about simplicial sets, which can be extracted from \cite{JoyalUABNotes}*{Proposition 2.12 and 2.13}:
\begin{propn}\label{propn:innerhorn}
  For all $0 < k < n$, the inclusion
  $\Delta^{n}_{\Seg} \hookrightarrow \Lambda^{n}_{k}$ has a filtration
  by pushouts of inner horn inclusions of dimension $< n$. \qed
\end{propn}

Applying \cite{enropd}*{Proposition 2.7.8} (or the fact that colimits
in presheaves are universal), we get:
\begin{cor}\label{cor:innerhornC}
  For all $[n](\bfc) \in \uC_{\otimes}$ and $0 < k < n$, the morphism
  \[\Delta^{n}_{\Seg}(\bfc) \hookrightarrow \Lambda^{n}_{k}(\bfc)\] has a filtration
  by pushouts of inner horn inclusions of dimension $< n$. \qed
\end{cor}

\begin{cor}\label{cor:iaissegeqC}
  Inner anodyne maps in  $\PSh(\uC_{\otimes})$ are Segal equivalences, and an object of
  $\PSh(\uC_{\otimes})$ is a Segal presheaf \IFF{} it is local with respect to the
  inner horn inclusions.
\end{cor}
\begin{proof}
  We first show by induction on $n$ that the inner horn inclusion
  $\Lambda^{n}_{k}(\bfc) \hookrightarrow \Delta^{n}(\bfc)$ is a Segal
  equivalence: This is immediate for $n = 2$, since
  $\Lambda^{2}_{1} = \Delta^{2}_{\Seg}$. Assuming we know the
  statement for inner horns of dimension $< n$, it follows from
  \cref{cor:innerhornC} that
  $\Delta^{n}_{\Seg}(\bfc) \hookrightarrow \Lambda^{n}_{k}(\bfc)$ is a
  Segal equivalence, and hence so is
  $\Lambda^{n}_{k}(\bfc) \hookrightarrow \Delta^{n}(\bfc)$ by the
  2-out-of-3 property. It follows that a Segal presheaf is necessarily
  local with respect to all inner horn inclusions, so suppose
  $X \in \PSh(\uC_{\otimes})$ has the latter property; then $X$ is
  also local with respect to all inner anodyne maps.  It is immediate
  from \cref{cor:innerhornC} that the spine inclusions are inner anodyne, so
  $X$ is then indeed a Segal presheaf.
\end{proof}

We also have a useful way to construct inner anodyne maps in $\PSh(\uC_{\otimes})$ from ones in $\PSh(\simp)$. To state this we first need to recall some definitions from \cite{enropd}:
\begin{defn}\label{def simple}
  For
  $X \in \PSh(\uC_{\otimes})$ and $F \in \PSh(\simp)$, we say a morphism
  $X \to p^*(F)$ in $\PSh(\uC_{\otimes})$ is \emph{simple} if for every map $\sigma \colon \Delta^{n} \to F$ in $\PSh(\simp)$ we have that in the pullback
  \[
    \begin{tikzcd}
      \sigma^{*}X \arrow{r} \arrow{d} &  X \arrow{d} \\
      p^{*}\Delta^{n} \arrow{r}{p^{*}\sigma} & p^{*}F,
    \end{tikzcd}
  \]
  the presheaf $\sigma^{*}X$ is representable and the adjoint map
  $p_{!}\sigma^{*}X \to \Delta^{n}$ is an equivalence (so
  $\sigma^{*}X \simeq \Delta^{n}(\bfc)$ for some $\bfc \in \uC^{\times n}$).
\end{defn}

\begin{ex}
  By \cite{enropd}*{Remark 2.7.12}, the counit map
  $\Delta^{n}(\bfc) \to p^{*}\Delta^{n}$ is simple for all $\bfc$, as
  is any pullback of such a map.
\end{ex}

By applying \cite{enropd}*{Lemma 2.7.14} to the set of inner
horn inclusions in $\PSh(\simp)$, we obtain the following:
\begin{propn}\label{propn pb of inner anodyne}
  If $f\colon X\to p^*L$ is a simple map in $\PSh(\uC_{\otimes})$ and
  $K\to L$ is an inner anodyne map in $\PSh(\simp)$, then its base
  change $p^{*}K \times_{p^{*}L}X \to X$ is inner anodyne in
  $\PSh(\uC_{\otimes})$ (and so in particular a Segal equivalence, by
  \cref{cor:iaissegeqC}). \qed
\end{propn}

\section{Tensor products and colimits: the presheaf case}\label{sec:tenspsh}
Let $\uC$ and $\uD$ be monoidal \icats{}, with corresponding cartesian
fibrations $p \colon \uC_{\otimes} \to \simp$ and
$q \colon \uD_{\otimes} \to \simp$. In this section we will use the
preceding discussion of inner anodyne maps to prove that the external
product
\[ \PSeg(\uC_{\otimes}) \times \PSeg(\uD_{\otimes}) \to \PSeg((\uC \times \uD)_{\otimes})\]
preserves colimits in each variable.

\begin{propn}\label{propn:productsimplex}
  Given a simplex $\alpha = (\sigma,\tau) \colon \Delta^{k} \to \Delta^{n} \times \Delta^{m}$ and objects $[n](\bfc) \in \uC^{\otimes}, [m](\bfd) \in \uD^{\otimes}$, we have a pullback square
    \[
    \begin{tikzcd}
      \Delta^{k}(\sigma^{*}\bfc, \tau^{*}\bfd) \arrow{r} \arrow{d} & \Delta^{n}(\bfc) \boxtimes \Delta^{m}(\bfd) \arrow{d} \\
      r^{*}\Delta^{k} \arrow{r} & r^{*}(\Delta^{n} \times \Delta^{m}),
    \end{tikzcd}
  \]
  where $r$ denotes the cartesian fibration
  $r := p \times_{\simp} q \colon \uC_{\otimes}
  \times_{\simp}\uD_{\otimes} \to \simp$.
\end{propn}
\begin{proof}
  We first compute
  \[
    \begin{split}
      \Map(\Delta^{l}(\bfc',\bfd'), & \,\Delta^{n}(\bfc) \boxtimes \Delta^{m}(\bfd)) \\
                                   & \simeq \left(\Delta^{n}(\bfc) \boxtimes \Delta^{m}(\bfd)\right)([l](\bfc', \bfd')) \\
                                                                                 &  \simeq \Delta^{n}(\bfc)([l](\bfc')) \times \Delta^{m}(\bfd)([l](\bfd')) \\
                                                                                 & \simeq \Map_{\uC_{\otimes}}([l](\bfc'), [n](\bfc)) \times \Map_{\uD_{\otimes}}([l](\bfd'), [m](\bfd)) \\
                                                                                 & \simeq \left(\coprod_{\phi \colon [l] \to [n]} \prod_{i=1}^{l}\Map_{\uC}(c'_{i},(\phi^{*}\bfc)_{i})\right) \\
       & \phantom{\simeq} \times \left(\coprod_{\psi \colon [l] \to [m]} \prod_{i=1}^{l}\Map_{\uD}(d'_{i},(\psi^{*}\bfd)_{i})\right) \\
       & \simeq \coprod_{\substack{\phi \colon [l] \to [n], \\ \psi \colon [l] \to [m]}}   \prod_{i=1}^{l}\Map_{\uC}(c'_{i},(\phi^{*}\bfc)_{i}) \times \Map_{\uD}(d'_{i},(\psi^{*}\bfd)_{i}).
    \end{split}
    \]
    From this, it is easy to identify $\Map(\Delta^{l}(\bfc',\bfd'), \alpha^{*}(\Delta^{n}(\bfc) \boxtimes \Delta^{m}(\bfd)))$ as
    \[ \coprod_{\gamma \colon [l] \to [k]} \prod_{i=1}^{l}\Map_{\uC_{\otimes}}(c'_{i}, (\gamma^{*}\sigma^{*}\bfc)_{i}) \times \Map_{\uD_{\otimes}}(d'_{i}, (\gamma^{*}\tau^{*}\bfd)_{i}),\]
    which is naturally equivalent to $\Map(\Delta^{l}(\bfc',\bfd'), \Delta^{k}(\sigma^{*}\bfc, \tau^{*}\bfd))$, as required.
  \end{proof}

\begin{cor}\label{propn:prodsimple}
  For $[n](\bfc) \in \uC^{\otimes}, [m](\bfd) \in \uD^{\otimes}$, the
  map
  \[\Delta^{n}(\bfc) \boxtimes \Delta^{m}(\bfd) \to r^{*}(\Delta^{n}
    \times \Delta^{m})\] is simple. \qed
\end{cor}

\begin{observation}
  We have a natural equivalence
  \[
    \begin{split}
      \Map(\Delta^{k}(\bfc,\bfd), p^{*}\Delta^{n} \boxtimes q^{*}\Delta^{m}) & \simeq (p^{*}\Delta^{n} \boxtimes q^{*}\Delta^{m})([k](\bfc,\bfd)) \\
                                                                             & \simeq(p^{*}\Delta^{n})([k](\bfc)) \times (q^{*}\Delta^{m})([k], \bfd) \\
                                                                             &  \simeq \Map_{\simp}([k], [n]) \times \Map_{\simp}([k], [m]) \\
       & \simeq \Map(\Delta^{k}(\bfc,\bfd), r^{*}(\Delta^{n} \times \Delta^{m})),
    \end{split}
  \]
  so that $p^{*}\Delta^{n} \boxtimes q^{*}\Delta^{m} \simeq r^{*}(\Delta^{n} \times \Delta^{m})$.
\end{observation}

\begin{cor}\label{cor:boxpoia}
  For all $[n](\bfc) \in \uC_{\otimes}, [m](\bfd) \in \uD_{\otimes}$ and $0 < k < m$, the pushout-product
  \[ \partial \Delta^{n}(\bfc) \boxtimes \Delta^{m}(\bfd) \amalg_{\partial \Delta^{n}(\bfc) \boxtimes \Lambda^{m}_{k}(\bfd)} \Delta^{n}(\bfc) \boxtimes \Lambda^{m}_{k}(\bfd) \to \Delta^{n}(\bfc) \boxtimes \Delta^{m}(\bfd)\]
  is inner anodyne, and so a Segal equivalence.
\end{cor}
\begin{proof}
  We know that the inclusion $\partial \Delta^{n} \times \Delta^{m} \amalg_{\partial \Delta^{n} \times \Lambda^{m}_{k}} \Delta^{n} \times \Lambda^{m}_{k} \to \Delta^{n} \times \Delta^{m}$ is inner anodyne in $\PSh(\simp)$; see for example the explicit filtration by pushouts of inner horn inclusions in \cite{DuggerSpivakMap}*{Lemma A.1}. We claim that we have a pullback square
  \[
    \begin{tikzcd}
      \partial \Delta^{n}(\bfc) \boxtimes \Delta^{m}(\bfd) \amalg_{\partial \Delta^{n}(\bfc) \boxtimes \Lambda^{m}_{k}(\bfd)} \Delta^{n}(\bfc) \boxtimes \Lambda^{m}_{k}(\bfd)  \arrow{r} \arrow{d} & \Delta^{n}(\bfc) \boxtimes \Delta^{m}(\bfd) \arrow{d} \\
      r^{*}(\partial \Delta^{n} \times \Delta^{m} \amalg_{\partial \Delta^{n} \times \Lambda^{m}_{k}} \Delta^{n} \times \Lambda^{m}_{k}) \arrow[hookrightarrow]{r} & r^{*}(\Delta^{n} \times \Delta^{m}).
    \end{tikzcd}
  \]
  Indeed, this holds because pullbacks preserve colimits (since
  presheaves form an $\infty$-topos), and $\boxtimes$ takes diagrams that
  are pullbacks in both variables to pullbacks by
  \cref{obs:boxlimcolim}. We can therefore conclude by \cref{propn pb of
    inner anodyne} that the top horizontal map is again inner anodyne,
  as required.
\end{proof}

\begin{cor}\label{cor:boxia}
  If a morphism $X' \to X$ in $\PSh(\uC_{\otimes})$ is inner anodyne, then
  \[ X' \boxtimes Y \to X \boxtimes Y\] is a Segal equivalence in $\PSh((\uC \times \uD)_{\otimes})$ for all $Y \in \PSh(\uD_{\otimes})$.
\end{cor}
\begin{proof}
  Since $\boxtimes$ on presheaves preserves colimits in each variable, and Segal equivalences are closed under colimits, it suffices to consider $Y$ of the form $\Delta^{m}(\bfd)$. Similarly, in the first variable it suffices to consider the case where we have a single pushout
  \[
    \begin{tikzcd}
      \Lambda^{n}_{k}(\bfc) \arrow{r} \arrow{d} & \Delta^{n}_{k}(\bfc) \arrow{d} \\
      X' \arrow{r} & X,
    \end{tikzcd}
  \]
  and applying $ \blank \boxtimes \Delta^{m}(\bfd)$ to this we see that it is furthermore enough to show that $\Lambda^{n}_{k}(\bfc) \boxtimes \Delta^{m}(\bfd) \to \Delta^{n}(\bfc) \boxtimes \Delta^{m}(\bfd)$ is a Segal equivalence. For this we induct on $m$, noting that for $m=0$ the formula for mapping spaces in the proof of \cref{propn:prodsimple} implies that we get $\Lambda^{n}_{k}(\bfc, \bbone) \to \Delta^{n}_{k}(\bfc, \bbone)$ where $\bbone$ is the unit in $\uD$. Assuming we have proved the equivalence in dimensions $< m$, decomposing the boundary of $\Delta^{m}$ as a colimit of lower-dimensional simplices implies that
  \[ \Lambda^{n}_{k}(\bfc)\boxtimes \partial\Delta^{m}(\bfd) \to \Delta^{n}_{k}(\bfc) \boxtimes \partial\Delta^{m}(\bfd)\]
  is a Segal equivalence. Our map for $\Delta^{m}$ decomposes as a composition
  \[ \Lambda^{n}_{k}(\bfc) \boxtimes \Delta^{m}(\bfd) \to
    \Lambda^{n}_{k}(\bfc) \boxtimes \Delta^{m}(\bfd)
    \amalg_{\Lambda^{n}_{k}(\bfc)\boxtimes \partial\Delta^{m}(\bfd)}
    \Delta^{n}_{k}(\bfc) \boxtimes \partial\Delta^{m}(\bfd) \to
    \Delta^{n}_{k}(\bfc) \boxtimes \Delta^{m}(\bfd).\] Here the first
  map is a cobase change of the map we just saw was a Segal
  equivalence, while the second map is inner anodyne by
  \cref{cor:boxpoia}.
\end{proof}

\begin{cor}\label{cor:boxsegeq}
  The exterior product
  \[ \blank \boxtimes \blank \colon \PSh(\uC_{\otimes}) \times  \PSh(\uD_{\otimes}) \to \PSh(\uC_{\otimes} \times_{\simp} \uD_{\otimes})\]
  preserves Segal equivalences in each variable.
\end{cor}
\begin{proof}
  Since $\boxtimes$ preserves colimits in each variable, it suffices to show that $\Delta^{n}_{\Seg}(\bfc) \boxtimes Y \to \Delta^{n}(\bfc) \boxtimes Y$ is a Segal equivalence for all $\bfc$ and $Y$. This follows from \cref{cor:boxia} since $\Delta^{n}_{\Seg}(\bfc) \to \Delta^{n}(\bfc)$ is inner anodyne.
\end{proof}

\begin{cor}\label{cor:segpshprodcolim}
  The exterior product of Segal presheaves
  \[ \blank \boxtimes \blank \colon \PSeg(\uC_{\otimes}) \times  \PSeg(\uD_{\otimes}) \to \PSeg(\uC_{\otimes} \times_{\simp} \uD_{\otimes})\]
  preserves colimits in each variable.
\end{cor}
\begin{proof}
  From \cref{lem:boxpresseg} we have a commutative square
  \[
    \begin{tikzcd}
      \PSeg(\uC_{\otimes}) \times \PSeg(\uD_{\otimes}) \arrow[hookrightarrow]{d} \arrow{r}{\boxtimes} & \PSeg((\uC \times \uD)_{\otimes}) \arrow[hookrightarrow]{d} \\
      \PSh(\uC_{\otimes}) \times \PSh(\uD_{\otimes})
      \arrow{r}{\boxtimes} & \PSh((\uC \times \uD)_{\otimes}).
    \end{tikzcd}
  \]
  If we write $L_{\Seg}$ for the localization functors $\PSh(\blank) \to \PSeg(\blank)$, left adjoint to the fully faithful inclusions $i_{\Seg} \colon \PSeg(\blank) \hookrightarrow \PSh(\blank)$,
  we get a mate transformation
  \[ L_{\Seg}(\blank \boxtimes \blank) \to L_{\Seg}(\blank) \boxtimes L_{\Seg}(\blank),\]
  which for objects $X, Y$ fits in a commutative triangle
  \[
    \begin{tikzcd}
      {} & X \boxtimes Y \arrow{dl} \arrow{dr} \\
      L_{\Seg}(X \boxtimes Y) \arrow{rr} & & L_{\Seg}X \boxtimes L_{\Seg}Y.
    \end{tikzcd}
  \]
  Here the left diagonal arrow is a Segal equivalence, and by \cref{cor:boxsegeq} so is the right diagonal arrow. It follows that the bottom arrow is an equivalence, since it is a Segal equivalence between Segal presheaves. Now recall that a colimit in Segal presheaves is computed by taking the colimit in presheaves and then applying $L_{\Seg}$. Given a diagram $\phi \colon \uK \to \PSeg(\uC_{\otimes})$ and $Y \in \PSeg(\uD_{\otimes})$, we get
  \[
    \begin{split}
(\colim_{\uK} \phi) \boxtimes Y & \simeq L_{\Seg}(\colim_{\uK} i_{\Seg}\phi) \boxtimes L_{\Seg}i_{\Seg}Y \\ &  \simeq L_{\Seg}((\colim_{\uK} i_{\Seg}\phi) \boxtimes i_{\Seg} Y) \\ & \simeq  L_{\Seg}(\colim_{\uK} (i_{\Seg}\phi \boxtimes i_{\Seg} Y)) \\ & \simeq \colim_{\uK} L_{\Seg}(i_{\Seg} \phi \boxtimes i_{\Seg} Y) \\ & \simeq \colim_{\uK} (\phi \boxtimes Y),
    \end{split}
  \]
  so that $\boxtimes$ indeed preserves colimits of Segal presheaves in each variable.
\end{proof}

Combining this with \cref{cor:psegprodinalgd}, we get:
\begin{cor}\label{cor:algdpshcolim}
  Suppose $\uC$ and $\uD$ are small monoidal \icats{}. Then the composite
  \[ \Algd(\PSh(\uC)) \times \Algd(\PSh(\uD)) \xto{\boxtimes} \Algd(\PSh(\uC) \times \PSh(\uD)) \xto{\mu_{*}} \Algd(\PSh(\uC \times \uD))\]
  preserves colimits in each variable.\qed
\end{cor}

\section{Tensor products and colimits: the presentable case}\label{sec:tenspres}
We now want to generalize \cref{cor:algdpshcolim} from presheaves to more general presentable \icats{}. Recall that a \emph{presentably monoidal} \icat{} is a monoidal \icat{} $\uV^{\otimes}$ such that the underlying \icat{} $\uV$ is presentable and the tensor product preserves colimits in each variable. In this case we can find a presentation of $\uV$ as a localization of a presheaf \icat{} that is compatible with the monoidal structure:
\begin{defn}
  Let $\uC$ be a small monoidal \icat{}. Then we say a set
  $\mathbb{S}$ of morphisms in $\PSh(\uC)$ is \emph{compatible with
    the monoidal structure} if the Day convolution tensor product of a
  morphism in $\mathbb{S}$ with an identity morphism lies in the
  strongly saturated class $\overline{\mathbb{S}}$ generated by
  $\mathbb{S}$. (For convenience, we will assume the set $\mathbb{S}$
  always includes the equivalences in $\uC$.) In this case, the full
  subcategory $\PSh^{\mathbb{S}}(\uC)$ of $\mathbb{S}$-local
  presheaves is a monoidal localization of $\PSh(\uC)$ by
  \cite{patterns2}*{Corollary 7.20}, meaning that the full subcategory
  $\PSh^{\mathbb{S}}(\uC)^{\otimes} \subseteq \PSh(\uC)^{\otimes}$,
  spanned by lists of $\mathbb{S}$-local presheaves, is a monoidal
  \icat{}, the inclusion is lax monoidal, and its left adjoint is
  strong monoidal.
\end{defn}

\begin{propn}\label{propn:presmonisloc}
  Suppose $\uV$ is a presentably monoidal \icat{}. Then there exists a small full monoidal subcategory $\uC$ of $\uV$ and a set $\mathbb{S}$ of morphisms in $\PSh(\uC)$ that is compatible with the monoidal structure, such that the Yoneda extension $\PSh(\uC) \to \uV$ induces an equivalence of monoidal \icats{}
  \[ \PSh^{\mathbb{S}}(\uC) \simeq \uV.\]
\end{propn}
\begin{proof}
  This follows from \cite{patterns2}*{Corollary 7.16} (which says that we can take $\uC$ to be the full subcategory of $\kappa$-compact objects for some regular cardinal $\kappa$).
\end{proof}

We can also identify the cocomplete tensor product in terms of such presentations:
\begin{lemma}\label{lem:cocomptenslocn}
  Suppose we have presentations $\uV \simeq \PSh^{\mathbb{S}}(\uC)$, $\uW \simeq \PSh^{\mathbb{T}}(\uD)$ of presentable \icats{}. Then their cocomplete tensor product is given by
  \[ \uV \otimes \uW \simeq \PSh^{\mathbb{S} \odot \mathbb{T}}(\uC \times \uD),\]
  where $\mathbb{S} \odot \mathbb{T}$ can be taken to consist of
  $\mathbb{S} \times \id_{d}$ for $d \in \uD$ together with $\id_{c}\times \mathbb{T}$ for $c \in \uC$.
  The canonical map $\uV \times \uW \to \uV \otimes \uW$ is given by the composition
  \[ \PSh^{\mathbb{S}}(\uC) \times \PSh^{\mathbb{T}}(\uD) \hookrightarrow   \PSh(\uC) \times \PSh(\uD) \xto{\mu} \PSh(\uC \times \uD) \to
    \PSh^{\mathbb{S} \odot \mathbb{T}}(\uC \times \uD).\]
\end{lemma}
\begin{proof}
  This follows from the proof of \cite{HA}*{Proposition 4.8.1.15}.
\end{proof}

\begin{observation}\label{obs:loctenscomp}
  The inclusion $\PSh^{\mathbb{S}}(\uC) \hookrightarrow \PSh(\uC)$ and its left adjoint induce an adjunction
  \[ \Algd(\PSh(\uC)) \rightleftarrows
    \Algd(\PSh^{\mathbb{S}}(\uC)),\] where the right adjoint is again
  fully faithful, so that we can identify
  $\Algd(\PSh^{\mathbb{S}}(\uC))$ as a full subcategory of
  $\Algd(\PSh(\uC))$.
  In the situation of \cref{lem:cocomptenslocn}, we then have a commutative diagram
  \[
    \begin{tikzcd}
      \Algd(\PSh^{\mathbb{S}}(\uC)) \times \Algd(\PSh^{\mathbb{T}}(\uD)) \arrow{r} \arrow[hookrightarrow]{d} & \Algd(\PSh^{\mathbb{S}}(\uC) \times \PSh^{\mathbb{T}}(\uD)) \arrow{r} \arrow[hookrightarrow]{d} &  \Algd(\PSh^{\mathbb{S} \odot \mathbb{T}}(\uC \times \uD)) \\
      \Algd(\PSh(\uC)) \times \Algd(\PSh(\uD)) \arrow{r}  &\Algd(\PSh(\uC) \times \PSh(\uD)) \arrow{r} &  \Algd(\PSh(\uC \times \uD)). \arrow{u}
    \end{tikzcd}
  \]
  Our goal is to show that the composite in the top row preserves
  colimits in each variable. To deduce this, we need to understand this diagram in terms of Segal presheaves.
\end{observation}

\begin{propn}
  The equivalence $\Algd(\PSh(\uC)) \simeq \PSeg(\uC_{\otimes})$ restricts to an equivalence
  \[ \Algd(\PSh^{\mathbb{S}}(\uC)) \simeq \PSeg^{\mathbb{S}}(\uC_{\otimes}),\]
  where $\PSeg^{\mathbb{S}}(\uC_{\otimes})$ is the full subcategory of Segal presheaves $X$ such that the restricted presheaf $X|_{\uC^{\op}}$ lies in $\PSh^{\mathbb{S}}(\uC)$.
\end{propn}
\begin{proof}
  This is a special case of \cite{patterns3}*{Proposition 2.5.10}.
\end{proof}

Next, we want to describe the full subcategory $\PSeg^{\mathbb{S}}(\uC_{\otimes})$ as a localization of $\PSeg(\uC_{\otimes})$. This requires introducing some notation:
\begin{observation}
  Suppose $p \colon \uE \to \uB$ is a cocartesian fibration. Given $b \in \uB$ and a functor $\phi \colon \uE_{b} \to \Spc$, consider the left Kan extension $i_{b,!}\phi \colon \uE \to \Spc$ along the fibre inclusion $i_{b} \colon \uE_{b} \to \uE$. Its value at $x \in \uE$ is a colimit over $\uE_{b/x} := \uE_{b} \times_{\uE} \uE_{/x}$, which maps via $p$ to $\{b\} \times_{\uB} \uB_{/b'} \simeq \Map_{\uB}(b,b')$, where $b' := p(x)$. The functor $\uE_{b/x} \to \Map_{\uB}(b,b')$ is a cocartesian fibration (since its target is an \igpd{}), and its fibre at $\beta \colon b \to b'$ can be identified with the fibre product $E_{b} \times_{E_{b'}} E_{b'/x}$ along the cocartesian transport functor $\beta_{!} \colon E_{b} \to E_{b'}$. We can thus rewrite the colimit formula for $i_{b,!}\phi(x)$ in two steps, as
  \[ \colimP_{\beta \in \Map_{\uB}(b,b')} \colimP_{(y, \beta_{!}y \to x)} \phi(y).\]
  Now we apply this to the cocartesian fibration $(\uC_{\otimes})^{\op} \to \Dop$:
Given $\bfphi = (\phi_{1},\ldots,\phi_{n})$ where $\phi_{i} \in \PSh(\uC)$, we write \[\Delta^{n}(\bfphi) := i_{[n],!}\Phi\]
where $\Phi := \phi_{1} \times \cdots \times \phi_{n} \colon (\uC_{\otimes}^{\op})_{[n]} \simeq (\uC^{\op})^{\times n} \to \Spc$. This presheaf satisfies
\[ \Delta^{n}(\bfphi)([m](\bfc)) \simeq \colimP_{\alpha \in \Map_{\simp}([m],[n])} \colimP_{(\bfd, \bfc \to \alpha^{*}\bfd)} \phi_{1}(d_{1}) \times \cdots \times \phi_{n}(d_{n}). \]
Here the inner colimit is over $(\uC^{\times n} \times_{\uC^{\times m}} \uC^{\times m}_{\bfc/})^{\op}$ along the cartesian transport functor $\alpha^{*}$, which we can identify with $\prod_{i=1}^{m} (\uC^{\times n_{i}} \times_{\uC} \uC_{c_{i}/})^{\op}$, where each fibre product is over an iterated tensor product. Using that the cartesian product commutes with colimits in each variable, we can rewrite the inner colimit in terms of the Day convolution tensor product as
\[ \prod_{i=1}^{m} \colimP_{(\bfd, c_{i} \to \bigotimes d_{j})} \prod_{j} \phi_{j}(d_{j}) \simeq \prod_{i=1}^{m} \left(\bigotimes_{\alpha(i-1) < j \leq \alpha(i)} \phi_{j}\right)(c_{i}).\]
In other words, we have pullback squares
\begin{equation}
  \label{eq:mapdeltaphi}
  \begin{tikzcd}
    \prod_{i=1}^{m} \Map_{\PSh(\uC)}(c_{i}, \bigotimes_{\alpha(i-1) < j \leq \alpha(i)} \phi_{j})
    \arrow{r} \arrow{d} & \Map_{\PSh(\uC_{\otimes})}(\Delta^{m}(\bfc), \Delta^{n}(\bfphi)) \arrow{d} \\
    \{\alpha\} \arrow{r} & \Map_{\simp}([m], [n]).
  \end{tikzcd}
\end{equation}
\end{observation}

\begin{observation}\label{obs:mapoutdeltaphi}
    From the definition of $\Delta^{n}(\bfphi)$ as a left Kan extension, we have
     \begin{equation}
      \label{eq:mapfromdnphi}
      \Map_{\PSh(\uC_{\otimes})}(\Delta^{n}(\bfphi), X) \simeq \Map_{\PSh(\uC^{\times n})}(\Phi, X|_{\uC^{\times n, \op}}),
\end{equation}
  where $\Phi$ is the presheaf $\phi_{1} \times \cdots \times \phi_{n}$. In particular, for $n = 1$ we have
  \[ \Map_{\PSh(\uC_{\otimes})}(\Delta^{1}(\bfphi), X) \simeq
    \Map_{\PSh(\uC)}(\phi, X|_{\uC^{\op}}).\] It follows that
  $\PSeg^{\mathbb{S}}(\uC)$ can be identified with the localization of
  $\PSeg(\uC)$ at $\Delta^{1}(\phi) \to \Delta^{1}(\psi)$ for
  $\phi \to \psi$ in $\mathbb{S}$.
\end{observation}

\begin{lemma}\label{lem:pbdayconv}
  Given $\alpha \colon [m] \to [n]$ let us write $\alpha^{*}\bfphi$ for the list $\otimes_{\alpha(i-1) < j \leq \alpha(i)} \phi_{j}$. Then we have pullback squares
  \[
    \begin{tikzcd}
      \Delta^{m}(\alpha^{*}\bfphi) \arrow{r} \arrow{d} & \Delta^{n}(\bfphi) \arrow{d} \\
      p^{*}\Delta^{m} \arrow{r} & p^{*} \Delta^{n}.
    \end{tikzcd}
  \]
\end{lemma}
\begin{proof}
  Follows from the description of mapping spaces in \cref{eq:mapdeltaphi}.
\end{proof}

\begin{lemma}\label{lem:delnphicolim}
  For presheaves $\phi_{i} \in \PSh(\uC)$ we have
  \[ \Delta^{n}(\bfphi) \simeq \colim_{\bfc \in \prod \uC_{/\phi_{i}}} \Delta^{n}(\bfc).\]
\end{lemma}
\begin{proof}
  Let $\Phi = \phi_{1} \times \cdots \times \phi_{n}$, so that we have
  \[ \Map_{\PSh(\uC_{\otimes})}(\Delta^{n}(\bfphi), X) \simeq \Map_{\PSh(\uC^{\times n})}(\Phi, X|_{\uC^{\times n, \op}}) \]
  by \cref{obs:mapoutdeltaphi}.
  In $\PSh(\uC^{\times n})$, the presheaf $\Phi$ is a colimit of representables over $(\uC^{\times n})_{/\Phi}$; here $(\uC^{\times n})_{/\Phi} \to \uC^{\times n}$ is the right fibration for $\Phi$, which is the product $\prod_{i} \uC_{/\phi_{i}}$. We thus have
  \[ \Map_{\PSh(\uC_{\otimes})}(\Delta^{n}(\bfphi), X) \simeq \lim_{\bfc \in (\prod_{i} \uC_{/\phi_{i}})^{\op}} \Map(\Delta^{n}(\bfc), X),\]
  which by Yoneda implies the required colimit description of $\Delta^{n}(\bfphi)$.
\end{proof}

\begin{observation}\label{obs:segphi}
  If we define
  \[ \Delta^{n}_{\Seg}(\bfphi) := \Delta^{1}(\phi_{1}) \amalg_{\Delta^{0}} \cdots \amalg_{\Delta^{0}} \Delta^{1}(\phi_{n}),\] then we have a pullback square
  \[
    \begin{tikzcd}
      \Delta^{n}_{\Seg}(\bfphi) \arrow{r} \arrow{d} & \Delta^{n}(\bfphi) \arrow{d} \\
      p^{*}\Delta^{n}_{\Seg} \arrow{r} & p^{*} \Delta^{n},
    \end{tikzcd}
  \]
  using \cref{lem:pbdayconv} and the fact that pullbacks preserve
  colimits in presheaves. On the other hand, applying the colimit description of $\Delta^{n}(\bfphi)$ from 
  \cref{lem:delnphicolim} to the same pullback square implies that the top horizontal morphism $\Delta^{n}_{\Seg}(\bfphi) \to \Delta^{n}(\bfphi)$ is the map on colimits
  \[ \colim_{\bfc \in \prod \uC_{/\phi_{i}}} \Delta^{n}_{\Seg}(\bfc) \to  \colim_{\bfc \in \prod \uC_{/\phi_{i}}} \Delta^{n}(\bfc).\]
  It follows that any Segal presheaf on $\uC_{\otimes}$ is also local with respect to $\Delta^{n}_{\Seg}(\bfphi) \to \Delta^{n}(\bfphi)$.
\end{observation}

\begin{lemma}\label{lem:deltanlocaleq}
  A Segal presheaf $X \in \PSh^{\mathbb{S}}(\uC_{\otimes})$ is also local with respect to the map $\Delta^{n}(\bfphi) \to \Delta^{n}(\bfpsi)$ induced by a list of morphisms $\phi_{i} \to \psi_{i}$ in $\mathbb{S}$.
\end{lemma}
\begin{proof}
  We have a commutative square
  \[
    \begin{tikzcd}
      \Delta^{n}_{\Seg}(\bfphi) \arrow{d} \arrow{r} & \Delta^{n}_{\Seg}(\bfpsi) \arrow{d} \\
      \Delta^{n}(\bfphi) \arrow{r}& \Delta^{n}(\bfpsi).
    \end{tikzcd}
  \]
  Here  $X$ is local with respect to the top horizontal morphism since it is a colimit of maps with respect to which $X$ is local by assumption. Moreover, $X$ is local with respect to the vertical morphisms by
  \cref{obs:segphi} since $X$ is a Segal presheaf. It follows that $X$ is indeed
  also local with respect to the bottom horizontal morphism.
\end{proof}

\begin{propn}
  For every simplex $\alpha = (\sigma,\tau) \colon \Delta^{k} \to \Delta^{n} \times \Delta^{m}$, we have pullback squares
  \[
    \begin{tikzcd}
      \Delta^{k}(\sigma^{*}\bfphi, \tau^{*}\bfpsi) \arrow{r} \arrow{d} & \Delta^{n}(\bfphi) \boxtimes \Delta^{m}(\bfpsi) \arrow{d} \\
      r^{*}\Delta^{k} \arrow{r} & r^{*}(\Delta^{n} \times \Delta^{m}).
    \end{tikzcd}
  \]
\end{propn}
\begin{proof}
  Compute mapping spaces as in the proof of \cref{propn:productsimplex}.
\end{proof}

We can then lift the description of $\Delta^{1}\times \Delta^{1}$ as a pushout $\Delta^{2} \amalg_{\Delta^{1}} \Delta^{2}$:
\begin{cor}\label{cor:boxdelta1po}
  We have a pushout
  \[ \Delta^{1}(\phi) \boxtimes \Delta^{1}(\psi) \simeq
    \Delta^{2}(\phi \times \bbone, \bbone \times \psi) \amalg_{\Delta^{1}(\phi \times \psi)} \Delta^{2}(\bbone \times \psi, \phi \times \bbone)
  \]
  in $\PSh((\uC \times \uD)_{\otimes})$ for all $\phi \in \PSh(\uC)$,
  $\psi \in \PSh(\uD)$. \qed
\end{cor}

\begin{cor}\label{cor:boxSeq}
  The external product $\PSeg(\uC) \times \PSeg(\uD) \to \PSeg(\uC \times \uD)$ takes $\mathbb{S}$-local Segal equivalences in the first variable and $\mathbb{T}$-local Segal equivalences in the second variable to $\mathbb{S} \odot \mathbb{T}$-local Segal equivalences.
\end{cor}
\begin{proof}
  Since the external product preserves colimits in each variable, it suffices to consider the case
  \[ \Delta^{1}(\phi) \boxtimes \Delta^{1}(\psi) \to \Delta^{1}(\phi')\boxtimes \Delta^{1}(\psi')\]
  for $\phi \to \phi'$ in $\mathbb{S}$ and $\psi \to \psi'$ in $\mathbb{T}$. Here it follows from the colimit decomposition in \cref{cor:boxdelta1po} and \cref{lem:deltanlocaleq} that this is indeed an $\mathbb{S} \odot \mathbb{T}$-local Segal equivalence.
\end{proof}

\begin{cor}
  The composite
  \[ \PSeg^{\mathbb{S}}(\uC_{\otimes}) \times \PSeg^{\mathbb{T}}(\uD_{\otimes})
    \hookrightarrow \PSeg(\uC_{\otimes}) \times \PSeg(\uD_{\otimes})
    \xto{\boxtimes} \PSeg((\uC \times \uD)_{\otimes})
    \to \PSeg^{\mathbb{S} \odot \mathbb{T}}((\uC \times \uD)_{\otimes})\]
  preserves colimits in each variable.
\end{cor}
\begin{proof}
  Let us write $L^{\mathbb{S}}$ for the localization $\PSeg(\uC_{\otimes}) \to \PSeg^{\mathbb{S}}(\uC_{\otimes})$ and $i^{\mathbb{S}}$ for its fully faithful left adjoint. We then want to show that the functor
  $L^{\mathbb{S} \odot \mathbb{T}}(i^{\mathbb{S}}(\blank) \boxtimes
  i^{\mathbb{T}}(\blank))$ preserves colimits in each variable. We
  consider the colimit of a diagram $p \colon \uK \to
  \PSeg^{\mathbb{S}}(\uC_{\otimes})$, which can be described as
  $L^{\mathbb{S}}(\colim_{\uK} i^{\mathbb{S}}p)$. By \cref{cor:boxSeq}
  we know that the map \[(\colim_{\uK} i^{\mathbb{S}}p) \boxtimes
    i^{\mathbb{T}}Y \to i^{\mathbb{S}}L^{\mathbb{S}}(\colim_{\uK}
    i^{\mathbb{S}}p) \boxtimes i^{\mathbb{T}}Y
    \simeq i^{\mathbb{S}}(\colim_{\uK} p) \boxtimes i^{\mathbb{T}}Y
  \] is an $\mathbb{S} \odot \mathbb{T}$-local Segal equivalence for any $Y$  in $\PSeg^{\mathbb{T}}(\uD_{\otimes})$. We therefore have
  \[
    \begin{split}
      L^{\mathbb{S} \odot \mathbb{T}}(i^{\mathbb{S}}(\colim_{\uK} p) \boxtimes i^{\mathbb{T}}Y) & \simeq
      L^{\mathbb{S} \odot \mathbb{T}}(\colim_{\uK} i^{\mathbb{S}}p \boxtimes i^{\mathbb{T}}Y) \\             & \simeq \colim_{\uK}                                                                                     L^{\mathbb{S} \odot \mathbb{T}}(i^{\mathbb{S}}p \boxtimes i^{\mathbb{T}}Y),
    \end{split}
  \]
  since $\boxtimes$ preserves colimits in each variable and $L^{\mathbb{S} \odot \mathbb{T}}$ is a left adjoint.
\end{proof}

Combining this with \cref{propn:presmonisloc},
\cref{lem:cocomptenslocn}, and \cref{obs:loctenscomp}, we get:
\begin{cor}\label{cor:prestenscolim}
  If $\uV$ and $\uW$ are presentably monoidal \icats{}, then the composite
  \[ \Algd(\uV) \times \Algd(\uW) \to \Algd(\uV \times \uW) \to \Algd(\uV \otimes \uW)\]
  preserves colimits in each variable. \qed
\end{cor}

Together with \cref{obs:compprescolim}, this gives:
\begin{cor}
  If $F \colon \uV \times \uW \to \uU$ is a monoidal functor between
  presentably monoidal \icats{} that preserves colimits in each
  variable, then
  \[ \Algd(\uV) \times \Algd(\uW) \xto{\boxtimes} \Algd(\uV \times \uW) \xto{F_{*}} \Algd(\uU)\]
  also preserves colimits in each variable.
\end{cor}

\section{Tensor products and colimits: the general case}\label{sec:tensgen}

We now want to extend \cref{cor:prestenscolim} to the case of general
cocomplete \icats{}; in other words, we want to drop the assumption of
presentability. For this we pass to a larger universe: let $\LSpc$ be
the (very large) \icat{} of large spaces, and write
$\LPSh(\uC) := \Fun(\uC^{\op}, \widehat{\Spc})$. If $\uC$ is
cocomplete, we get a fully faithful functor
$\uC \hookrightarrow \LPSh^{\mathbb{S}}(\uC)$ that preserves small
colimits, where $\mathbb{S}$ consists of the maps
$\colim_{\uK} y(p) \to y(\colim_{\uK} p)$ for a generating (large) set
of small diagrams $p$ in $\uC$.

If $\uD$ is another cocomplete \icat{}, with a corresponding embedding
$\uD \hookrightarrow \LPSh^{\mathbb{T}}(\uD)$, then the construction
of the cocomplete tensor product in \cite{HA}*{\S 4.8.1} shows that
the canonical functor $\uC \times \uD \to \uC \otimes \uD$ fits in a
commutative square
\[
  \begin{tikzcd}
    \uC \times \uD \arrow{r} \arrow[hookrightarrow]{d} & \uC \otimes \uD \arrow[hookrightarrow]{d} \\
    \LPSh^{\mathbb{S}}(\uC) \times \LPSh^{\mathbb{T}}(\uD) \arrow{r} & \LPSh^{\mathbb{S} \odot \mathbb{T}}(\uC \times \uD),
  \end{tikzcd}
\]
where the vertical maps are fully faithful and the right-hand one preserves small colimits. If $\uC$ and $\uD$ are monoidal \icats{} where the tensor products are compatible with small colimits, then the vertical maps are strong monoidal, and we obtain a commutative diagram
\[
  \begin{tikzcd}
    \Algd(\uC) \times \Algd(\uD) \arrow{r} \arrow[hookrightarrow]{d} & \Algd(\uC \otimes \uD) \arrow[hookrightarrow]{d} \\
    \Algd(\LPSh^{\mathbb{S}}(\uC)) \times \Algd(\LPSh^{\mathbb{T}}(\uD)) \arrow{r} & \Algd(\LPSh^{\mathbb{S} \odot \mathbb{T}}(\uC \times \uD)),
  \end{tikzcd}
\]
where the vertical maps are compatible with small colimits in the appropriate senses. Since the bottom horizontal map preserves (large) colimits in each variable by \cref{cor:prestenscolim} (applied in a larger universe), we conclude that the top horizontal map preserves small colimits in each variable. This proves the following:
\begin{cor}\label{cor:cocomptenscolim}
  If $\uV$ and $\uW$ are cocomplete monoidal \icats{} whose tensor products are compatible with small colimits, then the composite
  \[ \Algd(\uV) \times \Algd(\uW) \to \Algd(\uV \times \uW) \to \Algd(\uV \otimes \uW)\]
  preserves colimits in each variable. \qed
\end{cor}

Together with \cref{obs:compprescolim}, this gives:
\begin{cor}
  Let $\uV$ and $\uW$ be cocomplete monoidal \icats{} whose tensor
  products are compatible with small colimits. If
  $F \colon \uV \times \uW \to \uU$ is a monoidal functor that
  preserves colimits in each variable, then
  \[ \Algd(\uV) \times \Algd(\uW) \xto{\boxtimes} \Algd(\uV \times \uW) \xto{F_{*}} \Algd(\uU)\]
  also preserves colimits in each variable. \qed
\end{cor}

We can also deduce the corresponding results for complete algebroids:
\begin{cor}
  Let $\uV$ and $\uW$ be cocomplete monoidal \icats{} whose tensor products are compatible with small colimits.
  \begin{enumerate}[(i)]
  \item The composite
  \[ \Cat(\uV) \times \Cat(\uW) \to \Cat(\uV \times \uW) \to \Cat(\uV \otimes \uW)\]
  preserves colimits in each variable.
\item If
  $F \colon \uV \times \uW \to \uU$ is a monoidal functor that
  preserves colimits in each variable, then
  \[ \Cat(\uV) \times \Cat(\uW) \xto{\boxtimes} \Cat(\uV \times \uW) \xto{F_{*}} \Cat(\uU)\]
  also preserves colimits in each variable.
  \end{enumerate}
\end{cor}
\begin{proof}
  The first part follows from \cref{extprodcomp}(iii), by the same
  argument as in the proof of \cref{cor:segpshprodcolim}. We get the
  second part by combining this with \cref{obs:compprescolim}.
\end{proof}

\begin{cor}
  Suppose $\uV$ is a cocomplete $\uO \times \mathbb{E}_{1}$-monoidal
  \icat{} that is compatible with small colimits. Then the induced
  $\uO$-monoidal structures on $\Algd(\uV)$ and $\Cat(\uV)$ are also
  compatible with small colimits. \qed
\end{cor}

\begin{remark}
  If $\uV$ is a presentably $\mathbb{E}_{n+1}$-monoidal \icat{} (for
  $n \geq 1$) it now follows from the adjoint functor theorem that
  $\Algd(\uV)$ and $\Cat(\uV)$ are \emph{closed}
  $\mathbb{E}_{n}$-monoidal \icats{}. For $n \geq 2$ let
  $\FunV(\blank,\blank)$ denote the corresponding internal Hom on
  $\Algd(\uV)$. (We ignore the case $n = 1$ only because this requires
  considering different adjoints in the two variables.) It follows
  from \cite{enr}*{Proposition 5.5.9} that if $\eB$ is complete, then
  $\FunV(\eA, \eB)$ is also complete for any algebroid $\eA$;
  hence $\FunV(\blank, \blank)$ is also the internal Hom on
  $\Cat(\uV)$.
\end{remark}

\section{Essentially surjective and fully faithful functors}\label{sec:ffes}
In this section we first show that essentially surjective and fully
faithful morphisms form a factorization system on $\Cat(\uV)$, and
then prove that the tensor product (and internal Hom in the
presentable case) are compatible with this factorization system.

\begin{propn}\label{propn:esfffact}
  The essentially surjective and the fully faithful $\uc{V}$-functors
  form a factorization system on $\Cat(\uc{V})$. In other words,
  any $\uc{V}$-functor $F \colon \ec{C} \to \ec{D}$
  factors essentially uniquely as an essentially surjective functor
  followed by a fully faithful one.
\end{propn}

The proof relies on some preliminary results:
\begin{observation}
  Recall that a morphism in $\Spc$ is an \emph{epimorphism} if it is
  surjective on $\pi_{0}$ and a \emph{monomorphism} if it is an
  inclusion of connected components. These classes of maps form a
  factorization system on $\Spc$ \cite{HTT}*{Example 5.2.8.16}. Moreover, the
  epimorphisms are precisely the effective epimorphisms when we regard
  $\Spc$ as an $\infty$-topos.
\end{observation}

\begin{propn}\label{lem:pbgivesmono}
  Suppose $\uc{X}$ is an $\infty$-topos,
  $X_{\bullet},Y_{\bullet} \colon \Dop \to \uc{X}$ are groupoid
  objects, and $f \colon X_{\bullet} \to Y_{\bullet}$ is a map such
  that the commutative square
  \[
    \begin{tikzcd}
      X_{1} \arrow{r}{f_{1}} \arrow{d}[swap]{(d_{1},d_{0})} &
      Y_{1}\arrow{d}{(d_{1},d_{0})} \\
      X_{0}^{\times 2} \arrow{r}{f_{0}^{\times 2}} & Y_{0}^{\times 2}
    \end{tikzcd}
  \]
  is cartesian. Then the induced map on colimits
  \[ f_{-1} \colon X_{-1} := \colim_{\Dop} X_{\bullet} \to
  \colim_{\Dop} Y_{\bullet} =: Y_{-1}\] is a monomorphism.
\end{propn}
I thank Bastiaan Cnossen and Adrian Clough for explaining the
following proof. It is convenient to first spell out a key part of the argument separately:
\begin{observation}\label{obs:pbpastingtopos}
  Suppose $\uc{X}$ is an $\infty$-topos, and we have a commutative
  diagram
  \[
    \begin{tikzcd}
      A' \arrow{r} \arrow{d} & B' \arrow{r} \arrow{d} & C' \arrow{d} \\
      A \arrow[twoheadrightarrow]{r}{f} & B \arrow{r} & C
    \end{tikzcd}
  \]
  in $\uc{X}$, where the left square and the outer square are
  cartesian, and $f$ is an effective epimorphism. Then the right
  square is also cartesian: We want to show that the canonical map
  $B' \to C' \times_{C} B$ is an equivalence. Since pullback along an
  effective epimorphism is conservative by \cite{HTT}*{Lemma 6.2.3.16} and this map
  lives over $B$, we can check this after pulling back along $f$. But
  then since the left square is cartesian we get the canonical map
  $A' \to C' \times_{C} A$, which is an equivalence since the outer
  square is cartesian.
\end{observation}

\begin{proof}[Proof of \cref{lem:pbgivesmono}]
  We want to show that the diagonal $X_{-1} \to X_{-1} \times_{Y_{-1}}
  X_{-1}$ is an equivalence, or equivalently that the commutative square
  \[
    \begin{tikzcd}
      X_{-1} \arrow{r} \arrow{d}{\Delta} & Y_{-1} \arrow{d}{\Delta} \\
      X_{-1}^{\times 2} \arrow{r} & Y_{-1}^{\times 2}
    \end{tikzcd}
  \]
  is cartesian.

  Since $\uc{X}$ is an $\infty$-topos, the map $X_{0} \to X_{-1}$ is
  an effective epimorphism and the groupoid object $X_{\bullet}$ is
  effective, which means that the commutative square
  \[
    \begin{tikzcd}
      X_{1} \arrow{r}{d_{0}} \arrow{d}{d_{1}} & X_{0}\arrow{d} \\
      X_{0} \arrow{r} & X_{-1}
    \end{tikzcd}
  \]
  is cartesian, and similarly for $Y_{\bullet}$. Equivalently, the
  commutative square
  \[
    \begin{tikzcd}
      X_{1} \arrow{d}{(d_{1},d_{0})} \arrow{r} & X_{-1}
      \arrow{d}{\Delta} \\
      X_{0}^{\times 2}\arrow{r} & X_{-1}^{\times 2}
    \end{tikzcd}
  \]
  is cartesian, so that we have a commutative cube
  \[
    \begin{tikzcd}
      X_{1} \arrow{rr}\arrow{dr} \arrow{dd} & & Y_{1} \arrow{dd}
      \arrow{dr} \\
      & X_{-1} \arrow[crossing over]{rr}  & & Y_{-1} \arrow{dd} \\
      X_{0}^{\times 2} \arrow[crossing over]{rr} \arrow{dr}& & Y_{0}^{\times 2}
      \arrow{dr} \\
       & X_{-1}^{\times 2} \arrow{rr} \arrow[crossing over,leftarrow]{uu} & & Y_{-1}^{\times 2}
    \end{tikzcd}
  \]
  where the left, right and back faces are cartesian. Since
  $X_{0}^{\times 2} \to X_{-1}^{\times 2}$ is an effective
  epimorphism, the front face is also cartesian by \cref{lem:pbgivesmono}.
\end{proof}

\begin{cor}\label{lem:ffmono}
  Suppose $f \colon \ec{C} \to \ec{D}$ is a fully faithful morphism in
  $\Algd(\uc{V})$. Then $\iota \ec{C} \to \iota \ec{D}$ is a
  monomorphism in $\Spc$.
\end{cor}
\begin{proof}
  By \cite{enr}*{Lemma 5.3.5}, the commutative
  square
  \[
    \begin{tikzcd}
      \iota_{1} \ec{C} \arrow{r} \arrow{d} & \iota_{1} \ec{D}
      \arrow{d} \\
      \iota_{0} \ec{C} \times \iota_{0} \ec{C} \arrow{r} &       \iota_{0} \ec{D} \times \iota_{0} \ec{D}
    \end{tikzcd}
  \]
  is cartesian. Since $\iota_{\bullet} \ec{C}$ and $\iota_{\bullet}
  \ec{D}$ are groupoid objects in the $\infty$-topos $\Spc$, this
  implies that the map on colimits $\iota \ec{C} \to \iota \ec{D}$ is
  a monomorphism by \cref{lem:pbgivesmono}.
\end{proof}

\begin{lemma}\label{lem:fullsubcomp}
  Let $\ec{C}$ be a complete algebroid in $\uc{V}$, and let
  $f \colon X \to \iota \ec{C}$ be a monomorphism in $\Spc$. Then the
  cartesian transport $f^{*}\ec{C}$ in $\Algd(\uV)$ is also complete.
\end{lemma}
\begin{proof}
  Consider the commutative triangle
  \[
    \begin{tikzcd}
      X \arrow[twoheadrightarrow]{d} \arrow[hookrightarrow]{dr}{f} \\
      \iota(f^{*}\ec{C}) \arrow[hookrightarrow]{r} & \iota \ec{C}
    \end{tikzcd}
  \]
  Here the bottom horizontal map is a monomorphism by
  \cref{lem:ffmono}, while the left vertical map is an
  epimorphism. These maps must therefore give the unique
  epimorphism/monomorphism factorization of $f$. Since $f$ is by
  assumption a monomorphism, it follows that $X \to
  \iota(f^{*}\ec{C})$ must be an equivalence, \ie{} that $f^{*}\ec{C}$
  is complete. 
\end{proof}

\begin{proof}[Proof of \cref{propn:esfffact}]
  Since $\Algd(\uc{V}) \to \Spc$ is by definition a cartesian
  fibration, by \cite{HA}*{Proposition 2.1.2.5} we can lift the
  epi-/monomorphism factorization system on $\Spc$ to a factorization
  system on $\Algd(\uc{V})$ where the left class consists of all
  morphisms that map to epimorphisms in $\Spc$ and the right class
  consists of \emph{cartesian} morphisms over monomorphisms in $\Spc$,
  that is fully faithful morphisms given by monomorphisms on spaces of
  objects.  It follows from \cref{lem:fullsubcomp} that this restricts
  to a factorization system on the full subcategory
  $\Cat(\uc{V})$. The left class clearly consists of the essentially
  surjective morphisms in $\Cat(\uc{V})$, so it only remains to
  observe that the right class consists precisely of the fully
  faithful ones, since a fully faithful morphism between complete
  algebroids always lies over a monomorphism in $\Spc$ by \cref{lem:ffmono}.
\end{proof}

\begin{lemma}\label{lem:tensoresssurj}
  If $F \colon \ec{C} \to \ec{D}$ is an essentially
  surjective morphism in $\Cat(\uV)$, then so is \[F \otimes
  \id_{\ec{A}} \colon \ec{C} \otimes \ec{A} \to
  \ec{D} \otimes \ec{A}\] for any $\uc{V}$-\icat{}
  $\ec{A}$. 
\end{lemma}
\begin{proof}
  By construction of $\otimes$ on $\Cat(\uV)$,
  it is given by pushing forward the external tensor product in $\Cat(\uV \times \uV)$ along the tensor product functor $\uV \times \uV \to \uV$ and then completing. We therefore have a commutative square of
  \igpds{}
  \[
    \begin{tikzcd}
      \iota \ec{C} \times \iota \ec{A} \arrow{r} \arrow{d} &
      \iota \ec{D} \times \iota \ec{A} \arrow{d} \\
      \iota(\ec{C} \otimes \ec{A}) \arrow{r} &
      \iota(\ec{D} \otimes \ec{A})
    \end{tikzcd}
  \]
  where the vertical maps are surjective on $\pi_{0}$. Moreover, the
  top horizontal map is surjective on $\pi_{0}$ since $\pi_{0}$
  preserves products. It follows that the bottom horizontal map is
  also surjective on $\pi_{0}$, \ie{} $F \otimes \id_{\ec{A}}$
  is essentially surjective.
\end{proof}

\begin{cor}\label{cor:FunFF}
  Let $\uV$ be a presentably $\mathbb{E}_{n+1}$-monoidal \icat{} for $n \geq 2$. If
  a $\uc{V}$-functor $F \colon \ec{C} \to \ec{D}$ is a
  fully faithful morphism in $\Cat(\uV)$, then so is \[F_{*} \colon \FunV(\ec{A}, \ec{C}) \to
  \FunV(\ec{A}, \ec{D})\] for any $\uc{V}$-\icat{} $\ec{A}$.
\end{cor}

\begin{proof}
  Since essentially surjective and fully faithful
  $\uc{V}$-functors form a factorization system on $\Cat(\uV)$, it suffices to
  show that any commutative square
  \[
    \begin{tikzcd}
      \ec{I} \arrow{d}[swap]{G} \arrow{r} & \FunV(\ec{A},
      \ec{C}) \arrow{d}{F_{*}} \\
      \ec{J} \arrow{r} \arrow[dashed]{ur} & \FunV(\ec{A}, \ec{D})
    \end{tikzcd}
  \]
  where $G$ is essentially surjective, has an essentially unique
  diagonal filler. By adjunction, this is equivalent to showing there
  is an essentially unique filler in the square
    \[
    \begin{tikzcd}
      \ec{I} \boxtimes \ec{A} \arrow{d}[swap]{G \boxtimes \id_{\ec{A}}} \arrow{r} & \ec{C} \arrow{d}{F} \\
      \ec{J} \boxtimes \ec{A} \arrow{r} \arrow[dashed]{ur} &
      \ec{D}.
    \end{tikzcd}
  \]
  This follows since here the right vertical morphism is fully
  faithful and the left vertical morphism is essentially
  surjective  by \cref{lem:tensoresssurj}.
\end{proof}

\end{document}